\documentclass[reqno]{amsart}
\usepackage{graphicx} 
\usepackage[colorlinks,allcolors=darkblue]{hyperref}
\usepackage{amsmath}
\usepackage{mathtools}
\usepackage{amssymb}
\usepackage{amsthm}
\usepackage{enumitem}
\usepackage{bbm}
\usepackage{mhequ}
\usepackage{mathrsfs}
\usepackage{aliascnt}
\usepackage{microtype}
\usepackage{orcidlink}

\usepackage[nameinlink]{cleveref}

\colorlet{darkblue}{blue!70!black}
\colorlet{darkgreen}{green!50!black}

\def\dash{\leavevmode\unskip\kern0.18em--\penalty\exhyphenpenalty\kern0.18em}
\def\slash{\leavevmode\unskip\kern0.15em/\penalty\exhyphenpenalty\kern0.15em}

\newtheorem{thm}{Theorem}[section]

\newaliascnt{lem}{thm}       
\newtheorem{lem}[lem]{Lemma} 
\aliascntresetthe{lem}       

\newaliascnt{cor}{thm}
 
\aliascntresetthe{cor}

\newaliascnt{deff}{thm}
\newtheorem{deff}[deff]{Definition} 
\aliascntresetthe{deff}

\newaliascnt{rem}{thm}
\newtheorem{rem}[rem]{Remark} 
\aliascntresetthe{rem}

\newaliascnt{ass}{thm}
\newtheorem{ass}[ass]{Assumption} 
\aliascntresetthe{ass}

\crefname{lem}{Lemma}{Lemmas}
\Crefname{lem}{Lemma}{Lemmas}

\crefname{thm}{Theorem}{Theorems}
\Crefname{thm}{Theorem}{Theorems}

\crefname{cor}{Corollary}{Corollaries}
\Crefname{cor}{Corollary}{Corollaries}

\crefname{deff}{Definition}{Definitions}
\Crefname{deff}{Definition}{Definitions}

\crefname{rem}{Remark}{Remarks}
\Crefname{rem}{Remark}{Remarks}

\crefname{ass}{Assumption}{Assumptions}
\Crefname{ass}{Assumption}{Assumptions}

\numberwithin{equation}{section} 

\newcommand{\mat}[4]{%
  \begin{pmatrix}
    #1 & #2 \\
    #3 & #4
  \end{pmatrix}%
}

\def\invm{\mathrm{inv}}

\newcommand{\E}{\mathbb{E}}
\newcommand{\F}{\mathcal{F}}

\newcommand{\CC}{\mathcal{C}}

\newcommand{\Prb}{\mathbb{P}}
\newcommand{\mcM}{\mathcal{M}}
\newcommand{\ncN}{\mathcal{N}}
\newcommand{\Pp}{\mathcal{P}}
\newcommand{\N}{\mathbb{N}}
\newcommand{\R}{\mathbb{R}}
\newcommand{\Id}{\mathbbm{1}}
\newcommand{\inv}{\mathcal{M}_+}

\newcommand{\K}{\mathcal{K}}
\newcommand{\C}{\mathbb{C}}

\newcommand{\Ll}{\mathcal{L}}

\newcommand{\T}{\mathbb{T}}

\newcommand{\Z}{\mathbb{Z}}

\newcommand{\CD}{\mathcal{D}}

\def\fg{\mathfrak{g}}
\def\so{\mathfrak{so}}
\def\CH{\mathcal{H}}
\def\eps{\epsilon}
\def\P{\mathbf{P}}
\DeclarePairedDelimiter{\scal}{\langle}{\rangle}

\makeatletter 
\DeclareRobustCommand*{\act}{%
  \mathbin{\mathpalette\bigcdot@{}}%
}
\DeclareRobustCommand*{\place}{%
  \mathord{\kern0.1em\mathpalette\bigcdot@{}\kern0.1em}%
}
\newcommand*{\bigcdot@scalefactor}{.5}
\newcommand*{\bigcdot@widthfactor}{1.15}
\newcommand*{\bigcdot@}[2]{%
  \sbox0{$#1\vcenter{}$}
  \sbox2{$#1\cdot\m@th$}%
  \hbox to \bigcdot@widthfactor\wd2{%
    \hfil
    \raise\ht0\hbox{%
      \scalebox{\bigcdot@scalefactor}{%
        \lower\ht0\hbox{$#1\bullet\m@th$}%
      }%
    }%
    \hfil
  }%
}
\makeatother

\newcommand{\CO}{\mathcal{O}}
\newcommand{\CA}{\mathcal{A}}
\newcommand{\CB}{\mathcal{B}}

\DeclareMathOperator{\id}{id}
\DeclareMathOperator{\supp}{supp}
\DeclareMathOperator{\ad}{ad}
\DeclarePairedDelimiter{\ip}{\langle}{\rangle}

\def\eqref#1{(\ref{#1})}

\title[Asymptotics of Lyapunov Exponents]{Asymptotics of Lyapunov Exponents and Phase Transitions for Fluids with Degenerate Forcing}
\author{Martin Hairer\orcidlink{0000-0002-2141-6561} and Declan Stacy\orcidlink{0009-0006-7660-948X}}
\address{EPFL, Switzerland and Imperial College London, UK}
\email{martin.hairer@epfl.ch}
\address{EPFL, Switzerland}
\email{declan.stacy@epfl.ch}
\date{\today}

\begin{document}

\begin{abstract}
In this paper we analyze the Lyapunov exponents of a slow-fast system where the slow component is an Ornstein--Uhlenbeck process which perturbs the linear evolution of a fast variable through a bilinear form. These naturally arise in many finite-dimensional models for turbulence such as Galerkin truncation of 2D Navier--Stokes, the Lorenz 96 system, and the Lorenz 63 system. Using our general results about the slow-fast system, we are able to prove phase transitions in the ergodicity of each of these models when degenerate stochastic forcing is applied: as a parameter 
(e.g.\ noise strength or viscosity) varies, the number of invariant measures of the system switches from one to several. We are also able to obtain precise asymptotics for the top Lyapunov exponent associated to the unstable invariant measure. The crux of our proof is using a Wiener chaos expansion to show that mass quickly transfers 
from stable modes to unstable ones.
\vspace{1em}

\end{abstract}
\maketitle
\thispagestyle{plain}

\tableofcontents

\section{Introduction}

There have been numerous studies on the stability of almost-sure invariant subsets of stochastic systems (see for example \cite{ecologicalGeneral,persistence, extinction, transverse-lyap,persistence-infinite}) with applications ranging from biology to physics \cite{ecologicalContinuous, ecologicalDiscrete,lorenz, space-time}. In these 
articles, the authors show that the instability of such a subset, which implies the existence of multiple invariant measures, is governed by the positivity of a certain ``transverse''
Lyapunov exponent. In some cases, one can show that this exponent transitions from negative to positive 
when tuning a parameter of the model under consideration, leading to an interesting phase transition. 
However, in most cases it is extremely difficult to compute the exponent explicitly, or even to just determine its
sign (even numerically this is challenging, see \cite{foldes-numerical-lyap}). In this paper we 
develop a method for computing the asymptotics of these Lyapunov exponents which simplifies and generalizes
the results of \cite{lorenz,transverse-lyap}.
Our focus will be on toy models for fluid dynamics, where the transition from uniqueness to 
non-uniqueness of invariant measures 
can be interpreted as a form of transition to turbulence \cite{physics-interpretation}.

\subsection{Motivation}

We are inspired by \cite{lorenz}, where the authors studied the Lorenz~63 system (the $3$-dimensional ODE leading 
to the well-known Lorenz attractor) with random forcing applied to the $z$-component, and computed explicit 
asymptotics for the transverse Lyapunov exponent in terms of the strength of the noise. At any intensity of
the driving noise, this system admits the $z$-axis as an invariant set and
its linearisation around this line is given by 
\begin{equ}
    dz = -z\,dt + \alpha\, dW_t \,\qquad
    dx = (A + Bz)x\,dt \,.
\end{equ}
Here $\alpha > 0$ is the strength of the noise and $A,B$ are fixed $2 \times 2$ matrices. In fact, when $\alpha$ is large, $z$ can be viewed as a ``slow'' variable and $x$ as a ``fast'' variable via a change of variables inspired by \cite[Remark~1.1]{BJA22}, see the derivation of \eqref{eq:lorenz-changed} below. Thus, the natural heuristic is that as $\alpha$ tends to infinity, the Lyapunov exponent of $x$ should be close to $\E[\lambda(A + BZ)]$, where $\lambda(M)$ is the largest real part of the eigenvalues of $M$ (in other words the top Lyapunov exponent of $dx = Mx\,dt$) and $Z$ is 
distributed according to the invariant measure for 
the Ornstein--Uhlenbeck process $z$. In \cite[Theorem 5.2]{lorenz} this heuristic is proven rigorously using 
the ergodicity of a cleverly constructed discrete-time Markov chain, but it is unclear how that proof 
technique would generalize.

In this paper we use entirely different techniques to give a general proof of the heuristic which clearly illustrates how the ``slow'' randomness introduced by $z$ forces $x$ to have the ``correct'' growth rate.
A slightly simplified version of our main result (\Cref{thm:main}) then reads as follows:
\begin{thm}\label{thm:informal}
    Let $z \in \R^d$, $x \in \R^n \setminus \{0\}$, $A, B_i \in \R^{n \times n}$,
    \begin{equation}\label{eq:intro-eq}
           \begin{aligned}
                dz^\epsilon(t) &= -\epsilon z^\epsilon(t)\,dt + \sqrt{\epsilon}\,dW_t \\
                dx^\epsilon(t) &= \Big(A + \sum_{i=1}^d B_i z_i^\epsilon(t)\Bigr)x^\epsilon(t)\,dt \eqqcolon A(z^\epsilon(t))x^\epsilon(t)\,dt \,,
            \end{aligned}
    \end{equation}
    and $\lambda^\epsilon \coloneqq \lim_{t \to \infty} t^{-1}\ln \|x^\epsilon(t)\|$. Let $Z$ be distributed 
    according to the law of the invariant measure of $z^\epsilon(t)$ (which is independent of $\epsilon$). Then, provided
    that \eqref{eq:intro-eq} satisfies a weak non-degeneracy condition on $\R^d \times (\R^n \setminus \{0\})$, we have
    $$\lim_{\epsilon \to 0} \lambda^\epsilon = \E[\lambda(A(Z))] \,.$$
\end{thm}

\begin{rem}\label{rem:weakCond}
The non-degeneracy condition mentioned in the statement is satisfied whenever
\eqref{eq:intro-eq} satisfies the parabolic Hörmander condition on $\R^d \times (\R^n \setminus \{0\})$, 
but it is strictly weaker and does not even imply uniqueness of the invariant measure at $\epsilon > 0$,
as we will see in \Cref{ex:degenerate} below. 
\end{rem}

Thus, one contribution of this paper is the study of Lyapunov exponents for multi-scale systems such as the one above. While this is interesting in its own right, it turns out that many fluid models including Lorenz 63 (the one studied in \cite{lorenz}), Lorenz 96, and (Galerkin truncation of) 2D Navier--Stokes have a form which involves only a linear and a bilinear part. Thus, when additive noise is introduced, the relevant dynamics governing the (transverse) Lyapunov exponent can be written exactly in the form above. As a direct application of \Cref{thm:informal}, we are able to show both precise asymptotics of the Lyapunov exponents of the aforementioned systems with truly degenerate forcing (not even hypoelliptic), as well as phase transitions where the invariant measure goes from being unique to non-unique (using the results in \cite{persistence, extinction}).

Another motivation for studying this problem is a long-standing conjecture given in 
\cite[Rem~2.4]{hairer-mattingly} about 
the 2D stochastic Navier--Stokes equations. In particular, when certain degenerate forcing is added to 2D 
Navier--Stokes on a torus and the viscosity is small enough, we expect there to be multiple invariant measures. 
In this article we show that this conjecture holds in a special case at the level of the Galerkin truncations, 
which points towards a positive answer. (See \Cref{sec:sec:2d-navier-stokes} for more discussion; we also show 
that the infinite-dimensional conjecture is false for one particular edge case.) For our second example (Lorenz 96 system), analogous results about the non-uniqueness of invariant measures are shown in \cite{transverse-lyap}, but the precise asymptotics of the Lyapunov exponents are new (they show $\epsilon^{-1}\lambda_\epsilon \to \infty$
while we show in \Cref{thm:lorenz96} that $\lambda_\epsilon$ converges to a strictly positive constant $c$ with an explicit expression). Also, our assumptions are much easier to verify than the criteria in \cite{transverse-lyap} (the authors use an entirely different proof method which forces one to study a certain Lie algebra). For our last example, namely the Lorenz 63 system, these results were already 
obtained in \cite{lorenz}, but we show that they are an easy corollary of our general results, and again we emphasize that our proofs are quite different.

\subsection{Proof Strategy}
Our proof strategy for \Cref{thm:informal} is based on the Wiener Chaos expansion of the solution to the 
linear system for the variable $x$ (which does however depend non-linearly on $z$ and therefore on the 
driving noise). Since $dx = A(z)x\,dt$, we expect that $x_t \approx e^{tA(z_0)}x_0$ when $z$ is slow. We 
want to show that $\|x_t\| \approx e^{\lambda(A(z_0))t}$, but this is of course not true for all $x_0$
(just take for $x_0$ an eigenvector corresponding to an eigenvalue $\lambda$ with $\Re \lambda < \lambda(A(z_0))$).

Looking at the projective process $v_t = x_t/\|x_t\|$, this is a consequence of the fact that every eigenvector
of $A(z_0)$ leads to a critical point for its dynamics. These critical points are all saddle points, except
for those corresponding to $\Re \lambda = \lambda(A(z_0))$. The crux of the proof is therefore to
show that the noise is strong enough to ensure that the projective process does not get
`trapped' by these saddle points. 

It is well known that if one adds non-degenerate noise of intensity $\epsilon$, then it takes a
time of order $\sim \ln \epsilon^{-1}$ to escape a neighborhood of an unstable invariant set 
(one where the Lyapunov exponent is positive, so we are getting exponential growth in the directions 
orthogonal to the invariant subspace). Then once we are far enough from this set, it should be quite slow (at least polynomial in $\epsilon^{-1}$ time) to return since the noise is only of strength $\epsilon$ and the subset is unstable. Thus, we spend a proportionally minuscule amount of time near this ``bad'' set where the growth rate is small.

In our context, this $\epsilon$ amount of noise mentioned above comes from the fact that $z_t$ is changing (take in this example $dz = -\epsilon z\,dt + \sqrt \epsilon dW_t$), so it is random, but this happens quite slowly (on a time scale of order $\epsilon^{-1}$). In particular, since $A(z)$ is affine in $z$, the variation of constants formula yields
\begin{equation}\label{eq:intro-first-chaos}
    x_t = e^{tA(z_0)}x_0 + \int_0^t e^{(t-s)A(z_0)}A'(z_0)e^{sA(z_0)}x_0(z_s - z_0)ds + R_1(t)
\end{equation}
where $R_1(t) = O(\sup_{s \leq t}|z_s - z_0|^2)$ is a remainder term. 

This expression suggests that even if we start with $x_0$ belonging to an eigenspace of $A(z_0)$ with ``small'' exponential growth
rate, $x_t$ will grow at rate $\lambda(A(z_0))$ provided that $A'(z_0)x_0$ has a non-zero component in one of the
eigenspaces with maximal growth rate.
In general, one may need to consider a higher order expansion or to differentiate the integrand with respect to $s$ (thereby 
introducing Lie brackets of $A'(z_0)$ with $A(z_0)$). In this article we impose a non-degeneracy assumption on $z \mapsto A(z)$ which is general but also quite easy to verify (see \Cref{as:nondegen}). In particular, one advantage of the Wiener chaos-based proof technique
used in this article is that it makes it clear what the ``correct'' non-degeneracy assumption should be.

A more detailed proof sketch in the general case goes as follows:
\begin{enumerate}
    \item Expand $x_t$ as in \eqref{eq:intro-first-chaos} up to order $K$, where $K$ is based on our non-degeneracy assumption on $z \mapsto A(z)$.
    \item Rewrite the expansion in Wiener chaos form $x_t = \sum_{k=0}^K X_k(t) + R_K(t)$, where $X_k(t)$ lives in the $k$th Wiener chaos (note that $z_t$ is OU so it is equal to its first order Wiener chaos expansion, so this is indeed possible).
    \item Use the non-degeneracy assumption to lower bound $\E[\|PX_K(t)\|^2]$, where $P$ is the projection onto the ``unstable'' modes and $t \sim \ln \epsilon^{-1}$. By hypercontractivity, this implies that $PX_K(t)$ has a nontrivial probability of becoming big.
    \item Prove a stronger upper bound on the remainder $R_K(t)$ (this is easy since it has an extra power of $z_t - z_0$ and $z_t$ is a slow process).
    \item Show that $\|Px_t\|$ being big cannot happen if $\|Pv_t\|$ is small for a long time, in other words we cannot keep having most of the mass on stable modes, which shows that some mass transfers from stable to unstable modes.
    \item Show that once $\|Pv_t\|$ is not too small, then it probably stays not too small for a time which is polynomial in $\epsilon^{-1}$ (intuitively this is because the unstable modes grow much faster than the stable ones).
    \item Using the above, conclude that as $\epsilon \downarrow 0$ (proportionally) $v_t$ spends almost no time near $Pv_t = 0$.
    \item Since all invariant measures for $(z_t,v_t)$ when $\epsilon = 0$ are supported on the sets $Pv_t = 0$ or $\|Pv_t\| = 1$, we conclude that the limiting invariant measures are fully supported on the unstable modes, which is equivalent to $x_t$ growing like $e^{t\lambda(A(z))}$.
\end{enumerate}

\subsection{Related Problems}
Although we do not borrow directly from their proof, similar ideas appear in \cite{Tommaso, Batchelor}, in which the authors study the behavior of projective processes of SPDEs with low viscosity. In fact, the study of projective processes of SPDEs in general was part of the motivation for this paper because if one hopes to apply the general ideas of \cite{persistence, persistence-infinite} to attack the infinite-dimensional conjecture in \cite{hairer-mattingly}, one would need to obtain bounds on the $H^1$ norm of the projective process. One idea for obtaining such bounds is to show that mass moves from high-frequency Fourier modes to lower ones, and this is essentially what 
we do here (albeit in a finite-dimensional setting).

Other methods for proving positivity of Lyapunov exponents have been explored in \cite{BJA22, BJA23, BJ24}, in particular in the hypoelliptic case. The authors use a proof by contradiction relying on bounding the norm of a limiting invariant measure (assuming the exponents are negative) in a negative regularity space, which relies on a uniform hypoellipticity assumption. This assumption can be quite difficult to verify, often requiring computer assistance \cite{BJA23, BJ24}, and due to the contradiction-style of proof, the resulting asymptotics the authors obtain on the top Lyapunov exponent are not sharp. We also note that the models we study differ from \cite{BJA22, BJA23, BJ24} in that we consider degenerate (non-hypoelliptic) noise, and so we are more interested in the question introduced above about non-uniqueness of invariant measures, and positivity of the \textit{transverse} Lyapunov exponent. On the other hand, the results on positivity of Lyapunov exponents in \cite{BJA22, BJA23, BJ24} are more general in the sense that they allow for more nonlinearity in the coefficients. Thus, our results and those in \cite{BJA22, BJA23, BJ24} should be seen as complementary rather than overlapping, and the methods employed are extremely different.

Finally, looking at the last step of the proof above, we see that \Cref{thm:informal} is also related to another general problem, namely: if one adds uniformly hypoelliptic noise of strength $\epsilon$ to an ODE with an unstable invariant manifold $\mcM_0$, then as $\epsilon \downarrow 0$, do all limit points $\mu$ of the invariant measures $\mu_\epsilon$ of the SDEs satisfy $\mu(\mcM_0) = 0$? Or, also related to our proof, is the time needed to exit a small neighborhood of $\mcM_0$ proportional to $\ln \epsilon^{-1}$? This sort of question has been studied in the elliptic case (see for example \cite{Berkeley, exit-times-1, exit-times-2, exit-times-3}), but it would be interesting to see a hypoelliptic version of this in general (note that \eqref{eq:intro-eq} is not elliptic, instead our proof shows how the noise propagates from $z$ into $v = x/\|x\|$).

\subsection{Outline}
The paper proceeds as follows. In \Cref{sec:setup} we provide a more thorough discussion and formal statement of \Cref{thm:informal}. In \Cref{sec:applications} we present the applications of our main theorem \Cref{thm:main} to Galerkin truncation of 2D Navier--Stokes, Lorenz 96, and Lorenz 63 systems with degenerate stochastic forcing. In \Cref{sec:main} we present some preliminary lemmas and use them to prove \Cref{thm:main}. In \Cref{sec:mass-moving} we prove the lemmas in \Cref{sec:main} using the mass-moving ideas sketched above. Some more technical lemmas are proved in \Cref{sec:appendix}. There we also present an important theorem (\Cref{thm:persistence-extinction}) relating the (transverse) Lyapunov exponent to the stability of an invariant subset, which is essentially a summary of the results in the manuscripts \cite{persistence, extinction} adapted to our applications, which we include for the convenience of the reader.

\section{Main Theorem}\label{sec:setup}

    Let $z_t$ be an OU process on $\R^d$. For simplicity, we assume that $z$ has the form
    \begin{equation}\label{eq:z-def}
        dz_t = -Dz_t\,dt + dW_t \,,
    \end{equation}
    where $W_t = (W^1_t,\dots,W^d_t)$ is a standard $d$-dimensional Brownian motion and $D = \operatorname{Diag}(d_1,\dots,d_d)$ is a diagonal matrix with positive entries ($d_i > 0$).
    \begin{rem}
        In general, our results also apply if $z$ solves
        \begin{equ}[e:OUgen]
        dz_t = -Cz_t\,dt + \Sigma dW_t
        \end{equ}
        where $W_t$ is a standard $m$-dimensional Brownian motion, $\Sigma \in \R^{d \times m}$ has rank $d$, and $C \in \R^{d \times d}$ has $\Re \lambda > 0$ for all eigenvalues $\lambda$ of $C$. We stick with the less general formulation above for ease of notation and because all of our examples arise naturally in the form \eqref{eq:z-def}. Also, we note that writing $\Sigma = \bar \Sigma U$ with $U$ orthogonal and $\bar\Sigma$ of the form $\bar\Sigma = (A |0)$ with $A  \in \R^{d \times d}$, we can 
        reduce \eqref{e:OUgen} to \eqref{eq:z-def} as soon as $A^{-1}CA$ is symmetric.
    \end{rem}
    
    Let $B: \R^d \times \R^n \to \R^n$ be a bilinear form, which we may write as $B(z,x) = \sum_{i=1}^d z_iB_ix$ for $B_i \in \R^{n \times n}$, and 
    let $A^\epsilon: \R^n \to \R^n$ be a family of linear maps indexed by $\epsilon \in [0,1]$ such that $\epsilon\mapsto A^\epsilon$ is continuous. For $\epsilon = 0$, we simply write $A$.
    \begin{deff}\label{def:matrix-notation}
        For $z \in \R^d$, $\epsilon \in [0,1]$ we define the $n$ by $n$ matrices
        $$A^\epsilon(z) \coloneqq A^\epsilon + \sum_{i=1}^d z_iB_i \quad \text{and} \quad A(z) \coloneqq A^0(z)\;.
        $$
    \end{deff}
    
    Let $x^\epsilon$ be the solution to the random ODE
    \begin{equation}\label{eq:x-def}
        dx^\epsilon(t) = A^\epsilon(z^\epsilon(t))x^\epsilon(t)dt \,,
    \end{equation}
    and define the processes
    \begin{equation}\label{eq:zv-def}
        z^\epsilon(t) \coloneqq z_{\epsilon t} \,, \quad r^\epsilon(t) \coloneqq \|x^\epsilon(t)\| \,, \quad \text{and} \quad v^\epsilon(t) \coloneqq x^\epsilon(t)\|x^\epsilon(t)\|^{-1} \,.
    \end{equation}

    Note that although $v^\epsilon$, which takes values on the unit sphere $S^{n-1}$ of $\R^n$, is technically not well-defined when $x^\epsilon = 0$, the evolution equations for the above system
    also make sense there by continuity.
    Indeed, the chain rule yields
    \begin{equation}\label{eq:zvr-system-def}
    \begin{aligned}
        dz &= -\epsilon Dz\,dt + \sqrt{\epsilon} \,dW_t \\
        dv &= [A^\epsilon(z)v - \ip{A^\epsilon(z)v,v}v]\,dt \\
        dr &= r\ip{A^\epsilon(z)v,v}\,dt \,,
    \end{aligned}
    \end{equation}
    which still makes sense when $r = 0$. Formally,
    \begin{lem}\label{lem:well-posed}
        For a given
        $$(z_0,v_0,r_0,\epsilon) \in \R^d \times S^{n-1} \times [0,\infty) \times [0,1] \,,$$
        the equation \eqref{eq:zvr-system-def} has a unique solution $(z^\epsilon(t),v^\epsilon(t),r^\epsilon(t))$ with $(z^\epsilon(0),v^\epsilon(0),r^\epsilon(0)) = (z_0,v_0,r_0)$. Furthermore, $x^\epsilon(t) \coloneqq r^\epsilon(t)v^\epsilon(t)$ solves \eqref{eq:x-def} and $z^\epsilon(t) = z_{\epsilon t}$, where $z_t$ solves \eqref{eq:z-def}.
    \end{lem}

    We are interested in the Lyapunov exponent associated to \eqref{eq:x-def} defined by
    \begin{equation}\label{eq:lyap-exp-def}
        \lambda_\epsilon \coloneqq \lim_{t \to \infty} \frac{1}{t}\ln \|x^\epsilon(t)\| \,,
    \end{equation}
    which for every initial condition exists (as a random variable) almost surely by the multiplicative ergodic theorem. Another way of viewing the Lyapunov exponents is via the invariant measures of the projective process $(z,v)$. Define the function
    \begin{equation}\label{eq:L-lnr-def}
        \Lambda^\epsilon(z,v) \coloneqq \ip{A^{\epsilon}(z)v,v} \,,
    \end{equation}
    which captures the infinitesimal rate of growth of $\ln \|x\|$. Using the $dr$ portion of \eqref{eq:zvr-system-def}, we have
    \begin{equation}\label{eq:dlnr}
        \frac{1}{t}\ln\|x^{\epsilon}(t)\| = \frac{1}{t}\ln\|x_0\| + \frac{1}{t}\int_0^t \ip{A^{\epsilon}(z^\epsilon(s))v^\epsilon(s),v^\epsilon(s)}ds \,.
    \end{equation}
    Ergodic theory tells us that the Lyapunov exponents are intimately related to the set $P_{\invm}^\epsilon$ of invariant measures for the process $(z^\epsilon(t),v^\epsilon(t))$ on $\R^d \times S^{n-1}$. Indeed, if $\mu \in P_{\invm}^\epsilon$ is ergodic then for $\mu$-a.e. $(z_0,v_0) \in \R^d \times S^{n-1}$
    \begin{equation}\label{eq:lyap-exp-equals-ergodic-average}
        \lambda_\epsilon = \int \Lambda^\epsilon(z,v)d\mu(z,v) \,.
    \end{equation}
    
    When $\epsilon = 0$, $P_{\invm}$ is quite large since $z$ becomes a constant process, so for every $z_0 \in \R^d$ there are invariant measures which give probability one to $\{z = z_0\}$. These measures must be supported on invariant subspaces of the operator $A(z_0)$ (projected onto the sphere). For example, if the aforementioned operator is diagonal with distinct real eigenvalues $\lambda_1>\dots>\lambda_n$ then the corresponding eigenvectors $v_1,\dots,v_n$ each give $\delta_{(z_0,v_i)} \in P_{\invm}$, where $\delta$ denotes the Dirac delta measure. Also, $\int \Lambda(z,v)d\delta_{(z_0,v_i)}(z,v) = \Lambda(z_0,v_i) = \lambda_i$, so that the different possible values of \eqref{eq:lyap-exp-def} are exactly the eigenvalues. A general version of this correspondence is given below in \Cref{lem:description-of-Pinv}. We aim to show that, under a suitable nondegeneracy condition, in the limit as $\epsilon \downarrow 0$ the largest eigenvalues always win. In the example above, we would expect that the measure $\delta_{(z_0,v_i)}$ does not appear for $i > 1$. Thus, although $P_{\invm}$ is quite large, many measures are not stable with respect to perturbations by $\epsilon$; only the ``fastest-growing'' measures (for each $z_0$) should be observed. We define the set of eigenvalues and the largest one as follows:
    \begin{deff}\label{def:sigma-lambda}
        Given a matrix $T \in \R^{n \times n}$ we write $\sigma(T)$ for its spectrum
        and we set
        $\lambda(T) \coloneqq \sup\{\Re \lambda \mid \lambda \in \sigma(T)\}$.
    \end{deff}

Given a square matrix $A$, we define
$\operatorname{ad}_{A} \colon M \mapsto [A,M]$.
We say that a subspace $W \subset \R^n$ is invariant for an $n\times n$ matrix $M$ if $MW \subset W$,
and $W$ is a nontrivial invariant subspace if furthermore $W \neq \{0\}$ and $W \neq \R^n$. For a set $S$ of $n\times n$ matrices, we say that $W$ is invariant for $S$ if it is invariant for all matrices in $S$.

    \begin{ass}\label{as:nondegen}
    There are no nontrivial invariant subspaces for
    $S \coloneqq \{\ad_A^k(B_i) \mid k \geq 0, i = 1,\dots,d\}$.
    \end{ass}
 
With all of these preliminaries in place, the main result of this article is as follows.

    \begin{thm}\label{thm:main}
        Suppose \Cref{as:nondegen} holds. Then, for every $\epsilon_n \downarrow 0$ and initial conditions $(z_n,x_n)$ (with $x_n\neq 0$) for the process $(z^{\epsilon_n}(t), x^{\epsilon_n}(t))$ given by \eqref{eq:zv-def}, one has
        \begin{equation}\label{eq:main-first}
            \lim_{n \to \infty} \lambda_{\epsilon_n} = \lambda^* \coloneqq \E[\lambda(A(Z))]\quad \text{a.s.}
        \end{equation}
        (see \eqref{eq:lyap-exp-def} and \Cref{def:sigma-lambda}), where the law of $Z$ is the invariant measure for \eqref{eq:z-def}.

        Additionally,
        \begin{equation}\label{eq:main-second}
            \lim_{\epsilon \downarrow 0} \inf_{\mu^\epsilon \in P_{\invm}^\epsilon}\int \Lambda^\epsilon(z,v)d\mu^\epsilon(z,v) = \lambda^* \,,
        \end{equation}
        where $P^\epsilon_{\invm}$ is the set of invariant measures for $(z^\epsilon(t),v^\epsilon(t))$ and $\Lambda^\epsilon$ is as in \eqref{eq:L-lnr-def}.
    \end{thm}
The proof of this theorem is postponed to Section~\ref{sec:main}.

\begin{rem}\label{ex:degenerate}
As already alluded to in \Cref{rem:weakCond}, \Cref{as:nondegen} is strictly weaker than 
the parabolic Hörmander condition. Indeed, writing $\fg$ for the Lie algebra 
generated by $S$ and the identity, \Cref{as:nondegen} states that the canonical representation
of $\fg$ on $\R^n$ is irreducible, while Hörmander's condition for the 
projective process $(z^\epsilon(t),v^\epsilon(t))$ (for any fixed $\epsilon > 0$) is equivalent to 
the possibly stronger property that $\fg x = \R^n$ for all $x \in \R^n \setminus \{0\}$.

Consider then the following example with $n=5$, $d=3$, and $A=0$.
We identify $\R^5$ with the space $H$ of traceless symmetric real $3\times 3$ matrices and we
pick a basis $\{\Omega_i\}_{i \le 3}$ for $\so(3)$.
For $x \in H$, we then set
\begin{equ}[e:bracket]
B_i x = [\Omega_i, x]\;.
\end{equ}
(Here $[\place,\place]$ denotes the commutator, recalling that both $\Omega_i$ and $x$ are linear maps
on $\R^3$.) Since we set $A=0$ and the $B_i$ are antisymmetric if we endow $H$ with the Frobenius norm,
the processes $x_t$ and $v_t$ coincide in this example, provided that $|x_0| = 1$.
On the other hand, the $B_i$ are the generators of conjugation with rotation matrices, so there is
an $SO(3)$-valued process $O_t$ such that
\begin{equ}
x_t = O_t x_0 O_t^{-1}\;.
\end{equ}
In particular, the spectrum of $x_t$ remains constant in time,
$\sigma(x_t) = \sigma(x_0)$, which immediately implies that 
$P^\epsilon_{\invm}$ contains infinitely many ergodic invariant measures.

To show that \Cref{as:nondegen} holds, it remains to note that the representation 
of $\so(3)$ on $H$ given by \eqref{e:bracket} is irreducible, since it is nothing but the standard
``spin $2$'' representation of dimension $5$. (It is also easy to see this ``by hand'' since every
$x \in H \setminus \{0\}$ can be diagonalised by an element of $SO(3)$ and the diagonal matrices obtained by ordering 
its eigenvalues by ascending and descending absolute value are linearly independent. The linear span of its 
$SO(3)$ orbit must therefore contain all traceless diagonal matrices and therefore all of $H$.)
In this example, \Cref{thm:main} is of course overkill since \eqref{eq:main-first} and \eqref{eq:main-second}
hold trivially ($A,B_i$ are all antisymmetric so all sides of \eqref{eq:main-first}, \eqref{eq:main-second} are $0$).
\end{rem}

\section{Applications}\label{sec:applications}
Below we use \Cref{thm:main} to give precise asymptotics for the Lyapunov exponents of several fluid models subjected to degenerate stochastic forcing, and as a corollary we deduce that they are positive under certain regimes. It is easy to show that these Lyapunov exponents become negative under other regimes. Combined with \Cref{thm:persistence-extinction} (see below), which implies non-uniqueness of invariant measures in the positive exponent regime and uniqueness in the negative regime, we show a bifurcation of the long-term behavior which depends on the viscosity (in the Navier--Stokes case) or the strength of the noise (in the Lorenz 96 and Lorenz 63 systems).

\subsection{2D Navier--Stokes}\label{sec:sec:2d-navier-stokes}
In \cite{hairer-mattingly}, the authors study stochastic versions of the deterministic 2D Navier--Stokes equations on the two-dimensional torus $\T^2 = [-\pi,\pi]^2$:
$$\partial_t u + (u \cdot \nabla)u = \nu \Delta u - \nabla \rho \,, \quad \operatorname{div} u = 0$$
where $u = (u_1,u_2): \T^2 \times [0,\infty) \to \R^2$, $\rho: \T^2 \to \R$, $\nu > 0$ is called the viscosity, 
the Laplacian $\Delta$ acts componentwise, and
$$(u \cdot \nabla)u \coloneqq (u_1 \partial_{1}u_1 + u_2\partial_{2}u_1, u_1 \partial_{1}u_2 + u_2\partial_{2}u_2) \;.$$
Since $u$ is divergence-free, it is completely determined by its curl
$$w \coloneqq \nabla \wedge u \coloneqq \partial_{2}u_1 - \partial_{1}u_2\,,$$
so we may instead study the equations for $w: \T^2 \times [0,\infty) \to \R$. Alternatively, we can study the Fourier modes $w_k: [0,\infty) \to \C$ of $w$ given for each $k \in \Z^2$ by
$$w_k(t) \coloneqq \frac{1}{4\pi^2}\int_{\T^2} w(x,t)e^{-ik \cdot x}dx \,,$$
which satisfy $w_k = \overline w_{-k}$.
As explained in \cite[Eqs~2.1--2.2]{hairer-mattingly}, we arrive at
$$dw_k = \Big[-\nu|k|^2w_k - \frac{1}{4\pi} \sum_{j+l=k}\ip{j^\perp,l}(|l|^{-2}-|j|^{-2})w_lw_j\Big]dt \,,$$
where $j^\perp = (j_1,j_2)^\perp = (j_2,-j_1)$ and we assume $w$ has $0$ mean, meaning $w_0 = 0$.

In \cite{hairer-mattingly}, additive noise is applied to some of the Fourier modes, and if enough modes are forced then the system has a unique invariant measure (\cite[Theorem 2.1]{hairer-mattingly}). In \cite[Theorem 2.3]{hairer-mattingly} the authors mention ways of adding degenerate forcing so that perhaps there is an additional invariant measure, and they conjecture that when the viscosity $\nu$ is small enough then this is indeed the case (\cite[Remark 2.4]{hairer-mattingly}). Although this is not true in general (see \Cref{rem:bad-case}), we show below that this is indeed the case, at least at the level of the Galerkin truncations, 
when the modes $\{\pm (5,0), \pm(0,5),\pm(3,4),\pm(4,3),\pm(-3,4),\pm(-4,3)\}$ are forced. Since we consider truncations
and we always set the zero mode to $0$, we will write 
\begin{equ}
\Z^2_{\star,N} = \{k \in \Z^2 \mid k \neq (0,0), \|k\|_\infty \leq N \}\;,
\end{equ}
for the set of modes that we track.
\begin{thm}\label{thm:navier-stokes}
    Let $N \in \N$ and consider the following SDE:
    \begin{equation}\label{eq:stoch-navier-stokes}
        dw_k = \Big[-\nu|k|^2w_k - \frac{1}{4\pi} \sum_{j+l=k}\ip{j^\perp,l}(|l|^{-2}-|j|^{-2})w_lw_j\Big]dt + 1_{|k| = 5}dW^k_t \,,
    \end{equation}
    where $w(t) \in H$, with $H$ being the real vector space defined by
    $$H \coloneqq \{w : \Z^2_{\star,N}\to\C \mid w_k = \overline w_{-k} \} \,,$$
    and $W^k_t$ are $\C$-valued Brownian motions which are independent except for the reality condition
   $W^k = \overline{W^{-k}}$. Let $H_I \coloneqq \{w \in H \mid  w_k = 0 \text{ for all } k \notin I\}$, where $I \coloneqq \{k \in \Z^2 \mid |k| = 5\}$. Then for all $N$ large enough there are $0 < \nu_* \le \nu^* < \infty$
    (possibly depending on $N$) such that:
    \begin{itemize}
        \item When $\nu < \nu_*$, \eqref{eq:stoch-navier-stokes} has multiple invariant measures, at least one of which gives $H_I$ measure $0$. Furthermore, for all $\delta > 0$ there is some $\eta > 0$ such that for all $w(0) \notin H_I$,
        \begin{equation}\label{eq:stoch-persistence-eq}
            \limsup_{t \to \infty} \frac{1}{t}|\{s \leq t \mid d(w(s), H_I) < \eta \}| < \delta \quad a.s.\,,
        \end{equation}
        where $|\cdot|$ denotes Lebesgue measure on $[0,t]$ and $d(w,H_I) \coloneqq \inf_{z \in H_I} \|w - z\|$ is the distance from $w$ to $H_I$.
        \item When $\nu > \nu^*$, \eqref{eq:stoch-navier-stokes} has a unique invariant measure, and it is supported on $H_I$. Furthermore, almost surely
        $$\limsup_{t \to \infty} \frac{1}{t}\ln d(w(t),H_I) < 0 \,.$$
    \end{itemize}
\end{thm}
\begin{rem}
    We also show that the top Lyapunov exponent $\lambda_\nu$ associated to \eqref{eq:stoch-navier-stokes} in the case $w(0) \in H_I$ satisfies $\lim_{\nu \to 0} \nu^{1/2}\lambda_\nu = C > 0$, with an explicit formula for $C$. In other words, the transverse Lyapunov exponent of the unstable invariant measure is of order $\nu^{-1/2}$. We did not formally state this result because similar statements are given in the sections on Lorenz 96 and Lorenz 63 systems (\Cref{thm:lorenz96,thm:lorenz}). The proof is contained in the proof of \Cref{thm:navier-stokes} below by using \eqref{eq:main-first} in place of \eqref{eq:main-second} and a relation analogous to \eqref{eq:alpha-epsilon-relation}.
\end{rem}

\begin{rem}
    We could have tried to prove the same result when forcing collinear Fourier modes $k$, and \Cref{thm:main} would still apply in essentially the same way (the only difference being that one would have to treat each orbit of $\Z^2$ under the actions of shifting by the forced modes $k$ separately). However, in the second case of \cite[Theorem 2.3]{hairer-mattingly} where the forced $k$ form a lattice we cannot immediately apply \Cref{thm:main} because in that case the dynamics on $H_I$ is nonlinear as opposed to the much more tractable Ornstein--Uhlenbeck process we assume here.
\end{rem}

\begin{rem}\label{rem:bad-case}
Not all choices of degenerate forcing display such a transition. For example, consider forcing only the modes $\{\pm(1,0),\pm(0,1)\}$ (this is \eqref{eq:stoch-navier-stokes} but with $1_{\{|k|=5\}}$ replaced by $1_{\{|k|=1\}}$) and consider the Lyapunov function $V(w) = \sum (1-|k|^{-2})|w_k|^2$. Since the 2D Navier--Stokes nonlinearity preserves both energy $\|u\|_{L^2}^2 = \sum |u_k|^2 = \sum |k|^{-2}|w_k|^2$ and enstrophy $\|w\|_{L^2}^2 = \sum |w_k|^2$, one has
    $$dV(w) = -2\nu \sum (|k|^2-1)|w_k|^2dt \leq -4\nu V(w)dt \,.$$
In particular, $\lim_{t \to \infty}V(w_t) = 0$ almost surely so that $w_t \to \{w \mid w_k = 0 \text{ for all } |k| \neq 1\}$. Since the only invariant measure on this set is the invariant measure of the OU process, we obtain 
unique ergodicity regardless of the viscosity $\nu$.
\end{rem}

\begin{rem}
    There is some $\kappa > 0$ such that  we may take $\eta = \delta^\kappa$ (for small enough $\delta$) in \eqref{eq:stoch-persistence-eq}. A sketch of proof is as follows. Setting $W = |z|^2 + \epsilon^{2/3}|x|^2$ we have (similar to \eqref{eq:bound-on-l-bar-u}) that $\Ll W \leq K - cW$ and $\Gamma W \lesssim W$, where $\Gamma$ denotes the carré du champ operator. As in \cite{ForthcomingWork}, one can then show that,
with $V$ agreeing with $-\ln \|x\|$ for $\|x\| \leq 1/2$, 
the quantity $F(X_t) \coloneqq V(X_t) + CW(X_t)$ (where $C > 0$ is a large constant and $X_t$ denotes the process $(z_t,x_t)$ below) decreases linearly fast over any time interval $[0,t]$, except on an event of probability at most $e^{-c't}$ for some fixed $c' > 0$ (independent of $t$). In particular, $\Prb_x(\inf_{t \lesssim F(x)}F(X_t) \gtrsim 1) \leq e^{-c'F(x)}$ and for $F(x) \lesssim 1$ we have $\Prb_x(\sup_{t \lesssim e^{c'A/2}} F(X_t) \geq A) \leq e^{-c'A/2}$. By an argument similar to that at the end of the proof of \Cref{lem:construction-of-mn}, this implies a bound of the form $\mu(F \leq M) \geq (1 - 2e^{-c'M/2})\frac{e^{c'M/2} - M}{e^{c'M/2}}\mu(F \leq 2M)$ for large enough $M$ and any invariant measure $\mu$ with $\mu(H_I) = 0$. Thus $\mu(F\geq M) \leq e^{-c'M/3}$, which implies $\mu(|x| \leq e^{-M})\leq e^{-c'M/3}$ and therefore proves the claim for $\kappa = c'/3$.
\end{rem}

\begin{rem}
Since we do not show that $\nu_* = \nu_*(N)$ can be taken independent of $N$,
it is not clear whether \Cref{thm:navier-stokes} still holds for the full Navier--Stokes system ($N=\infty$). Suppose for each $\nu$ that most of the mass in Fourier space is concentrated in modes $k$ with $|k| \approx N(\nu)$, so to mimic the true dynamics we would need to do a Galerkin truncation of order $N(\nu)$. If $\nu_*(N(\nu)) < \nu$, then \Cref{thm:navier-stokes} does not tell us anything.

\end{rem}

\begin{proof}
    Note that \eqref{eq:stoch-navier-stokes} can be rewritten in the form
    $$dw = [\nu Aw + B(w,w)]dt + d\boldsymbol{W}_t \,,$$
    where $A$ is a linear map, $B$ is a bilinear form, and $\boldsymbol{W}_t$ is a standard (real) $H_I$-valued Wiener process. We define $H_I^\perp \coloneqq \{w \in H \mid w_k = 0 \text{ for all } k \in I\}$. 
    
 
    We can separate $w = (\tilde z,\tilde x)$, where $\tilde z$ is the $H_I$-component of $w$ and $\tilde x$ is the $H_I^\perp$-component. Then abusing notation and using $B = (B_I, B_{I^\perp})$, $\overline B(v,w) = B(v,w) + B(w,v)$ we have
    \begin{equation*}
        \begin{aligned}
            d\tilde z &= [-25\nu\tilde z + \overline B_I(\tilde z,\tilde x) + B_I(\tilde x,\tilde x)]dt + d \boldsymbol{W}_t \\
            d\tilde x &= [\nu A\tilde x + \overline B_{I^\perp}(\tilde z,\tilde x) + B_{I^\perp}(\tilde x,\tilde x)]dt \,,
        \end{aligned}
    \end{equation*}
    where we use the fact that $B(\tilde z,\tilde z) = 0$. Then we make a change of variables again $(z_t,x_t) = (\nu^{1/2}\tilde z_{\nu^{1/2}t}, \tilde x_{\nu^{1/2}t})$ so we obtain
     \begin{equation*}
        \begin{aligned}
            dz &= [-25\nu^{3/2}z + \nu^{1/2}\overline B_I(z,x) + \nu B_I(x,x)]dt + \nu^{3/4}d \boldsymbol{W}_t \\
            dx &= [\nu^{3/2} Ax + \overline B_{I^\perp}(z,x) + \nu^{1/2}B_{I^\perp}(x,x)]dt \,,
        \end{aligned}
    \end{equation*}
    and setting $\epsilon \coloneqq \nu^{3/2}$ we obtain
         \begin{equation}\label{eq:navier-stokes-changed}
        \begin{aligned}
            dz &= [-25\epsilon z + \epsilon^
            {1/3}\overline B_I(z,x) + \epsilon^
            {2/3}B_I(x,x)]dt + \epsilon^{1/2}d \boldsymbol{W}_t \\
            dx &= [\epsilon Ax + \overline B_{I^\perp}(z,x) + \epsilon^{1/3}B_{I^\perp}(x,x)]dt \,.
        \end{aligned}
    \end{equation}
    We are now in the context of \Cref{thm:persistence-extinction} below and one can verify that $\bar U(z,x) \coloneqq e^{\frac{1}{2}|z|^2 + \frac{1}{2}\epsilon^{2/3}|x|^2}$ satisfies the assumptions. Indeed, \eqref{eq:Lyapunov-condition} is satisfied since
    \begin{equation}\label{eq:bound-on-l-bar-u}
        \begin{aligned}
            \Ll \bar U &= \bar U\Big[-25\epsilon |z|^2 + \epsilon\ip{Ax,\epsilon^{2/3}x} + \ip{(z,\epsilon^{1/3}x),B((z,\epsilon^{1/3}x),(z,\epsilon^{1/3}x))} + \frac{1}{2}\epsilon(|z|^2 + \dim H_I)\Big] \\
            &= \bar U\Big[-\frac{49}{2}\epsilon |z|^2 + \epsilon\ip{Ax,\epsilon^{2/3}x} + \frac{1}{2}\epsilon\dim H_I\Big] \\
            &\leq \epsilon \bar U\Big[-|z|^2 - \epsilon^{2/3}|x|^2 + \frac{1}{2}\dim H_I\Big] \\
            &\leq K - \frac{\epsilon}{2}\bar U\Big(\dim H_I + |z|^2 + \epsilon^{2/3}|x|^2\Big)\,,
        \end{aligned}
    \end{equation}
    
    where the second equality uses $\ip{y,B(y,y)} = 0$, the first inequality uses $\ip{Ax,x} \leq -|x|^2$, and
    $$K \coloneqq \sup_{\epsilon \in [0,1], |z|^2 + \epsilon^{2/3}|x|^2 \leq 2\dim H_I} \epsilon \bar U\Big[-\frac{1}{2}|z|^2 - \frac{\epsilon^{2/3}}{2}|x|^2 + \dim H_I\Big] \in (0,\infty) \,.$$
    (Indeed, if $a > 2\dim H_I$ then $-a + \dim H_I/2 \leq -a/2 - \dim H_I/2$.)
    Then \eqref{eq:vanishing-condition} is satisfied because
    \begin{align*}
        |\ip{\epsilon Ax + \overline B_{I^\perp}(z,x) + \epsilon^{1/3}B_{I^\perp}(x,x),x}| &\lesssim |x|^2(1 + |z| + \epsilon^{1/3}|x|) \\
        &\lesssim |x|^2\Big(1 + |z|^2 + \epsilon^{2/3}|x|^2\Big) \\
        &\lesssim |x|^2\Big(2K - \frac{\Ll \bar U}{\bar U}\Big) \,,
    \end{align*}
    where the last inequality is due to \eqref{eq:bound-on-l-bar-u} and $\lesssim$ means inequality up to a multiplicative constant independent of $(z,x)$. Also, \eqref{eq:accessibility-condition} is satisfied with $\Upsilon(z,x) \coloneqq |z|^2 + \epsilon^{2/3}|x|^2$ using an almost identical computation to \eqref{eq:bound-on-l-bar-u} which we leave to the reader. Following \Cref{thm:persistence-extinction} \eqref{eq:bigZ-v-process}, we define the process
    \begin{equation}\label{eq:final-stokes}
          \begin{aligned}
            dZ &= -25 \epsilon Zdt + \epsilon^{1/2}d\boldsymbol{W}_t \\
            dv &= [\epsilon Av + \overline B_{I^\perp}(Z,v) - \Lambda^\epsilon(Z,v)v]dt \,,
        \end{aligned}
    \end{equation}
    where
    $$\Lambda^\epsilon(Z,v) \coloneqq \ip[\Big]{\epsilon Av + \overline B_{I^\perp}(Z,v),v} \,,$$
    and the quantities
    \begin{align*}
     \Lambda^{+}(\epsilon) &\coloneqq \sup_{\mu^\epsilon \in P^\epsilon_{\invm}} \int \Lambda^\epsilon(z,v)d\mu^\epsilon(z,v) \\
        \Lambda^{-}(\epsilon) &\coloneqq \inf_{\mu^\epsilon \in P^\epsilon_{\invm}} \int \Lambda^\epsilon(z,v)d\mu^\epsilon(z,v) \,,
    \end{align*}
    where $P^\epsilon_{\invm}$ denotes the set of all invariant measures of $(Z_t,v_t)$ with $|v_t| = 1$. By \Cref{thm:persistence-extinction}, our \Cref{thm:navier-stokes} will follow if we show that
    \begin{equation}\label{eq:limits-of-lambda-plus-minus}
        \liminf_{\epsilon \to 0} \Lambda^-(\epsilon) > 0 \quad \text{and} \quad \limsup_{\epsilon \to \infty} \Lambda^+(\epsilon) < 0
    \end{equation}
    (our changes of variables do not affect the conclusions of either theorem and $\epsilon \to 0,\infty$ iff $\nu \to 0,\infty$). Since $\ip{Av,v} \leq -1$ and $\sup_{\mu^\epsilon \in P^\epsilon_{\invm}} \int \ip{\overline B_{I^\perp}(z,v),v}d\mu^\epsilon(z,v) \leq C$ where $C$ is independent of $\epsilon$ ($Z_t$ has the same invariant measure for all $\epsilon > 0$), we conclude $\limsup_{\epsilon \to \infty} \Lambda^+(\epsilon) = -\infty$.

    For the $\epsilon \to 0$ case, we use \Cref{thm:main} \eqref{eq:main-second} (note that \eqref{eq:final-stokes} agrees with \eqref{eq:zvr-system-def} with $A^\epsilon(z) = \epsilon A + \overline B_{I^\perp}(z,\cdot)$), so it remains to verify \Cref{as:nondegen} and show
    $$\lambda^* \coloneqq \E[\lambda(\overline B_{I^\perp}(Z,\cdot))] > 0$$
    where $Z$ has the invariant law and $\lambda(\cdot)$ is the top eigenvalue (see \Cref{def:sigma-lambda}). In the notation of \Cref{def:matrix-notation} we have $A = 0$ and $B_l = \overline B_{I^\perp}(e_l,\cdot)$, where $\{e_l\}_{l \in I}$ is the basis for the (real) vector space $H_I$ given by
    $$e_l = \begin{cases}
        f_l + f_{-l} & \text{if } l \in J \coloneqq \{l \in I \mid l_2 > 0 \text{ or } l = (l_1,0) \text{ for some } l_1 > 0\} \\
        if_{-l} - if_{l} & \text{if } l \notin J
    \end{cases}$$
    where $\{f_l\}_{l \in \Z^2_{\star,N}}$ is the standard (complex) basis of $\C^{\Z^2_{\star,N}}$.
    Note that $\operatorname{tr}(\overline B_{I^\perp}(z,\cdot)) = 0$ for all $z$, which implies $\lambda(\overline B_{I^\perp}(z,\cdot)) \geq 0$. Thus, we are done once we verify \Cref{as:nondegen} and show the existence of a $z$ such that $\lambda(\overline B_{I^\perp}(z,\cdot)) > 0$ (which implies it for a set of $z$ of positive measure by continuity of $\lambda$, see \Cref{lem:cty-of-eigenspaces}).

    To verify \Cref{as:nondegen}, first note that our operators $B_l$ can be viewed as operators acting on $H_I^\perp$.
    In particular, for $l \in J$ they have the form
    \begin{equation}\label{eq:bl-def}
        \begin{aligned}
            (B_lw)_k &= -\frac{1}{2\pi} \ip{k^\perp,l}\Big[(|l|^{-2} - |k-l|^{-2})w_{k-l} - (|l|^{-2} - |k+l|^{-2})w_{k+l}\Big] \\
            (B_{-l}w)_k &= -\frac{1}{2\pi} \ip{k^\perp,l}\Big[(|l|^{-2} - |k-l|^{-2})iw_{k-l} + (|l|^{-2} - |k+l|^{-2})iw_{k+l}\Big]
        \end{aligned}
    \end{equation}
    (see \eqref{eq:stoch-navier-stokes} and set $j = k - l$). Then we can rewrite
    \begin{equation}\label{eq:b-in-terms-of-cde}
        B_l = C_l + C_{-l} \quad \text{and} \quad B_{-l} = iC_l - iC_{-l}
    \end{equation}
    where
    \begin{equation}\label{eq:cde-def}
        \begin{aligned}
            (C_lw)_k &= c_{k,l} w_{k-l} \;,\qquad c_{k,l} = -\frac{\ip{k^\perp,l}}{2\pi}(5^{-2} - |k-l|^{-2}) \,.
        \end{aligned}
    \end{equation}
    Fix notation $S = \C^I$ and $S^\perp = \C^{\Z^2_{\star,N} \setminus I}$, which are naturally viewed as subspaces of $\C^{\Z^2_{\star,N}}$. Note that by \eqref{eq:b-in-terms-of-cde} if there is a real invariant subspace $W \subset H_I^\perp$ for $\{B_l \mid l \in I\}$ then $W + iW \subset S^\perp$ is a complex invariant subspace for $\{C_l \mid l \in I\}$. Thus, to verify \Cref{as:nondegen} it suffices to show that the collection $\{C_l \mid l \in I\}$, viewed as linear operators on $S^\perp$, 
admits no non-trivial invariant subspace. In fact we will show that, assuming $N$ is large enough, for any nonzero $v \in S^\perp$ and standard basis vector $e_k$, where $k \in \Z^2_{\star,N} \setminus I$, one can find some $K > 0$ and indices $\ell_1,\ldots,\ell_K$ such that 
    $C_{\ell_K} \cdots C_{\ell_1} v \propto e_k$. This implies that any non-zero invariant subspace
 must contain all the $e_k$'s and therefore all of $S^\perp$.

To see that such a product exists, fix $v = \sum_j c_j e_j$, where $c_j =0$ 
for $j\not\in \Z^2_{\star,N}\setminus I$. Note that one has
\begin{equ}
C_\ell e_j \propto e_{j+\ell}\;,
\end{equ}
and $C_\ell e_j = 0$ if and only if either $j+\ell \not \in \Z^2_{\star,N}\setminus I$ or $\ell$ and $j$ are parallel.
In other words, if we write $\supp v = \{j \in \Z^2_{\star,N}\,:\, c_j \neq 0\}$, we have the recursion
\begin{equ}
\supp C_\ell v = (\ell + \supp v) \cap \bigl(\Z^2_{\star,N}\setminus (I \cup \R \ell)\bigr)\;.
\end{equ}
Whenever $N$ is large enough,\footnote{It suffices that the square 
$(N-10,N]\times (N-10,N]$ does not intersect $\bigcup_{v \in I} \R v$, which is the case when $N > 40$.} 
we now claim that, for every $j \in \Z^2_{\star,N}\setminus I$, one can find a sequence $\ell_m$ such that 
$C_{\ell_K}\cdots C_{\ell_1} e_j \propto e_{(N,N)}$ and doesn't vanish.
First, we note that one can find a sequence
$\ell_m \in \{(\pm5,0),(0,\pm 5)\}$ 
such that $C_{\ell_K}\cdots C_{\ell_1} e_j \neq 0$ and such that $\hat \jmath \coloneqq j + \sum_m \ell_m \in (N-5,N]\times (N-5,N]$.
Writing $C^\uparrow = C_{(4,3)}C_{(-4,3)}C_{(0,-5)}$, and similarly for $C^\to$, we then note that 
$C^\uparrow e_{\hat \jmath} \propto e_{\hat \jmath + (0,1)}$ and doesn't vanish, which immediately yields the claim.
By choosing $j \in \supp v$ minimal
for the lexicographic order on $\Z^2$, this guarantees that $\supp (C^\to)^{N-\hat\jmath_1} (C^\uparrow)^{N-\hat\jmath_2} C_{\ell_K}\cdots C_{\ell_1} v = \{(N,N)\}$. Reverting 
these steps allows one to move $v$ to a multiple of $e_k$ for any $k \in \Z^2_{\star,N}\setminus I$.

    Finally, we need to show that there is some $z$ such that $\lambda(\overline B_{I^\perp}(z,\cdot)) > 0$. We will show that $B_{(5,0)}$ has a positive real eigenvalue, and thus the claim is true for $z = e_{(5,0)}$. We construct the eigenvector $w$ explicitly as follows. Let $K = \lfloor N/5 \rfloor$ and notice that the subspace
    $$\bigg\{w = \sum_{n=-K}^K y_ne_{(5n,1)}\,\bigg|\, y_n = (-1)^ny_{-n}\bigg\} \cong \{y = (y_0,\dots,y_K) \mid y_i \in \C\}$$
    is invariant for $B_{(5,0)} \eqqcolon T$, where
    $$(Ty)_n = \begin{cases}
        a_{n+1}y_{n+1} - a_{n-1}y_{n-1} & \text{if } n > 0 \\
        2a_1y_1 &\text{if } n = 0
    \end{cases}\,, \quad a_n \coloneqq \frac{1}{25} - \frac{1}{25n^2+1} \,.$$
    Thus, it suffices to find an eigenvector $y$ for $T$.
    
    
    To construct a candidate eigenvector with eigenvalue $\lambda$, we recursively define $y$ by setting 
    $y_K = 1$ and then
    \begin{equation}\label{eq:first-eval}
        y_n = a_n^{-1}(-\lambda y_{n+1} + a_{n+2}y_{n+2}) \,,
    \end{equation}
    where $y_{K+1} = 0$ by convention. This guarantees that $(Ty)_n = \lambda y_n$ for $n > 0$. To have $(Ty)_0 = \lambda y_0$ we need to check that
    \begin{equation}\label{eq:second-eval}
        -\lambda y_0 + 2a_1y_1 = 0 \,.
    \end{equation}
    For both the cases where $K$ is odd and $K$ is even we show that it is always possible to choose a positive real $\lambda$ for which \eqref{eq:second-eval} is true, which will finish the proof. First note that
    $$y_n =\sum_{k=0}^{K-n} c_{n,k} \lambda^k \,, \quad \text{with $c_{n,k} = 0$ whenever $k \equiv K -n-1 \pmod 2$}$$
    (a polynomial which is either odd or even depending on the parity of $K-n$) where the coefficients
    $c_{n,k}$ satisfy a recursion in terms of the $a_l$'s. One can verify that
    \begin{equ}[e:exprckn]
    c_{n,K-n} = (-1)^{K-n}a_n^{-1}a_{n+1}^{-1}\dots a_{K-1}^{-1} \quad \text{and for $K-n$ even } \quad c_{n,0} = a_n^{-1}a_K \,.
    \end{equ}
    Keep also in mind that $a_n > 0$ for $n > 0$, while $a_0 < 0$.
    
    The left-hand side of \eqref{eq:second-eval} is a polynomial $P_{K,a}(\lambda)$ with
    leading term given by $-c_{0,K}\lambda^{K+1}$.
    When $K$ is odd one has $c_{0,K}>0$ by \eqref{e:exprckn}, so the leading term is negative. 
    Furthermore $P_{K,a}(0) = 2a_1a_1^{-1}a_K > 0$ by \eqref{e:exprckn}, so $P_{K,a}$ 
    must indeed have at least one positive real root.

    When $K$ is even, $P_{K,a}$ is odd and has a positive leading order term, so it suffices to check that
    \begin{equ}[e:wantedBoundP']
    P_{K,a}'(0) = -c_{0,0} + 2a_1c_{1,1}  < 0 \,.
    \end{equ}
    Differentiating  \eqref{eq:first-eval} in $\lambda$ and evaluating at $\lambda = 0$ yields the recursion
    \begin{equ}
    a_n c_{n,1} = - c_{n+1,0} + a_{n+2} c_{n+2,1}\;,
    \end{equ}
    so that 
    \begin{equ}
    a_1 c_{1,1} = -\bigl(c_{2,0} + c_{4,0} + \dots + c_{K,0}\bigr) \le -1\;.
    \end{equ}
    Since $c_{0,0} = a_0^{-1}a_K$ with $a_0 = - \frac{24}{25}$ and $a_K < \frac{1}{25}$, the bound
    \eqref{e:wantedBoundP'} follows at once.
\end{proof}
\begin{rem}
    The last step of the proof could have also been done using the general results about Hamiltonian operators in \cite{hamiltonians} since $T$ is the composition of the diagonal operator $(y_n)_{|n| \leq K} \mapsto (a_ny_n)_{|n| \leq K}$ with the antisymmetric operator $(y_n)_{|n| \leq K} \mapsto (y_{n+1}-y_{n-1})_{|n| \leq K}$. This works since $a_0 < 0$ and the remaining coefficients are positive. For example, the case of $K$ odd follows from \cite[Cor.~2.5]{hamiltonians}. The case of $K$ even requires showing that \cite[Eq.~2.18]{hamiltonians} is negative, which is indeed the case, but we opted to write out the more explicit construction above instead.
\end{rem}

\subsection{Lorenz 96 System}
In \cite{transverse-lyap} the authors consider the Lorenz 96 system with weak degenerate forcing given by
\begin{equation}\label{eq:lorenz96}
    du_j = [-\epsilon u_j + (u_{j+1} - u_{j-2})u_{j-1}]dt + \sqrt{\epsilon}1_{j \in I}dW^j_t \,,
\end{equation}
where $K \in \N$, $K \geq 3$, $N = 3K$, $u(t) = (u_0(t), \dots, u_{N-1}(t)) \in \R^N$, $\epsilon > 0$, $I=\{3k\mid k=0,\ldots,K-1\}$, all indices are taken mod $N$ (for example $u_N = u_0$), and $W^j_t$ are independent standard Brownian motions. We also use their notation $u(t) = \phi_{\omega}^t(u(0))$, where $\omega \in \Omega$ is the realization of the noise and $u(0)$ is the initial condition, and $A^t_{u(0),\omega} = D_{u(0)}\phi_{\omega}^t$, the derivative of the flow at $u(0)$. Letting $\{e_j\}_{j=0}^{N-1}$ be the standard basis for $\R^N$, we let $H_I \coloneqq \operatorname{Span}\{e_j \mid j \in I\} \cong \R^K$ and $H_I^\perp \coloneqq \operatorname{Span}\{e_j \mid j \notin I\} \cong \R^{2K}$ so that $\R^N \cong H_I \oplus H_I^\perp$.

The multiplicative ergodic theorem (see \cite[Prop.~3.4]{transverse-lyap}) shows that the transverse Lyapunov exponent
$$\lambda^\perp_\epsilon \coloneqq \lim_{t \to \infty} \frac{1}{t} \ln \|A^t_{u(0),\omega}|_{H_I^\perp} \|$$
(where $A^t_{u(0),\omega}|_{H_I^\perp}$ is the restriction of $A^t_{u(0),\omega}$ to $H_I^\perp$) exists and is constant almost surely for $u(0) \in H_I$ (in the remainder of this section, we treat $\lambda_\epsilon^\perp$ as a real-valued constant as opposed to a random variable which depends on $u(0) \in H_I$). We also note that alternatively
\begin{equation}\label{eq:alternative-transverse-def}
    \lambda_\epsilon^\perp = \max_{i} \lim_{t \to \infty}\frac{1}{t}\ln \|A^t_{u(0),\omega}e_i\|
\end{equation}
where $e_i$ is a fixed basis for $H_I^\perp$. (Indeed, we have $\|A^t_{u(0),\omega}|_{H_I^\perp}\| \geq \|A^t_{u(0),\omega}e_i\|$ by definition, and, choosing a norm for which $e_i$ are orthonormal, which is valid since it does not affect the limits, we have $\|A^t_{u(0),\omega}|_{H_I^\perp} \| \leq \sqrt{\operatorname{dim}H_I^\perp} \max_i \|A^t_{u(0),\omega}e_i\|$. We are also free to swap the $\max_i$ and the $\lim_{t \to \infty}$ because the limits exist for each $i$ by the multiplicative ergodic theorem.)
In \cite[Lem.~3.6]{transverse-lyap} the authors use computer algebra to show that it satisfies
$$\lim_{\epsilon \downarrow 0} \frac{\lambda_\epsilon^\perp}{\epsilon} = \infty \,.$$

As an application of our results (no computer algebra required), we obtain the following sharper estimate on the transverse Lyapunov exponent:

\begin{thm}\label{thm:lorenz96}
For each $z \in \R^K$, let $A(z): \R^{2K} \to \R^{2K}$ be given by
\begin{equation}\label{eq:96-az-def}
    \begin{aligned}
        (A(z)x)_{2k} &= (x_{2k+1} - x_{2k-1})z_k \\
        (A(z)x)_{2k+1} &= (z_{k+1}-z_k)x_{2k}
    \end{aligned}
\end{equation}
(the indices for $x$ are taken modulo $2K$ and those for $z$ are taken modulo $K$) 
and define $\lambda(z)$ to be the largest real part of the eigenvalues of $A(z)$. Then
    \begin{equation}\label{eq:lorenz96-thm}
        \lim_{\epsilon \downarrow 0} \lambda_\epsilon^\perp = \E[\lambda(Z)] > 0 \,,
    \end{equation}
    where $Z$ is an $\R^K$-valued Gaussian random variable with mean $0$ and covariance matrix $\frac{1}{2}I$.
\end{thm}

\begin{rem}
    As a direct corollary of \Cref{thm:persistence-extinction}, as in \Cref{thm:navier-stokes} we may deduce the existence of multiple invariant measures when $\epsilon$ is small, and when $\epsilon$ is large enough there is convergence to a unique invariant measure on $H_I$.
\end{rem}

\begin{rem}
    In fact, the actual Lyapunov exponent
    $$\lambda_\epsilon \coloneqq \lim_{t \to \infty} \frac{1}{t} \ln \|A^t_{u(0),\omega}\|$$
    is equal to the transverse Lyapunov exponent for small enough $\epsilon$ because
    $$\lim_{t \to \infty} \frac{1}{t} \ln \|A^t_{u(0),\omega}|_{H_I} \| = -\epsilon$$
    and $H_I, H_I^\perp$ are invariant subspaces for $A^t_{u(0),\omega}$ when $u(0) \in H_I$ (see \cite[Lemma 3.2]{transverse-lyap}). Indeed, the linearization $w(t) \coloneqq A^t_{u(0)}w(0)$ of \eqref{eq:lorenz96} on $H_I$ will have no noise (because the noise was additive), and the nonlinearity will disappear on $H_I$, leaving simply $dw_{3k} = -\epsilon w_{3k}$. Thus, we have also shown that $\lim_{\epsilon \downarrow 0} \lambda_\epsilon = \E[\lambda(Z)] > 0$.
\end{rem}

\begin{proof}
First notice that the vector
$$v\coloneqq (0,1,0,1,\dots,0,1) \in \R^{2K} = H_I^\perp \,,$$
satisfies $A(z)v = 0$ for all $z \in \R^K$ (see \eqref{eq:96-az-def}). Thus, \eqref{eq:lorenz96} has a nontrivial invariant subspace $H_I \oplus \R v$, and when restricted to this subspace we see that the component in the direction of $v$ has quite trivial dynamics (made explicit below). Thus, it makes sense to separate $v$ out by decomposing the process $u$ into
\begin{align*}
    \tilde z^\epsilon(t) &\coloneqq (u_0(t),u_3(t),\dots,u_{3K-3}(t)) \\
   w^\epsilon(t) &\coloneqq (u_1(t),u_2(t),u_4(t),u_5(t),\dots,u_{3K-2}(t),u_{3K-1}(t)) \\
    \tilde x^\epsilon(t) &\coloneqq (\Id - P)w^\epsilon(t) \,, \quad Pw \coloneqq K^{-1}\ip{w,v}v \,, \quad y^\epsilon(t) \coloneqq K^{-1}\ip{w^\epsilon(t),v}
\end{align*}
and note that we use the indexing conventions $\tilde x^\epsilon(t)_0 = u_1(t)$, etc.\ with indices taken mod $2K$, $\tilde z^\epsilon(t)_j = u_{3j}(t)$, and $\Id$ is the identity operator. Rewrite \eqref{eq:lorenz96} as
\begin{equation}\label{eq:96-not-linearized}
    \begin{aligned}
        d\tilde z_j &= [-\epsilon \tilde z_j + (\tilde x_{2j} - \tilde x_{2j-2})(\tilde x_{2j-1} + y)]dt + \sqrt{\epsilon} dW_t^{3j} \\
        d\tilde x &= [-\epsilon \tilde x + (\Id - P)A(\tilde z)\tilde x]dt \\
        dy &= [-\epsilon y + K^{-1}\ip{A(\tilde z)\tilde x,v}] dt \,,
    \end{aligned}
\end{equation}
from which it follows that $\lambda_\epsilon^\perp = (-\epsilon) \vee \tilde \lambda_\epsilon$, where $\tilde \lambda_\epsilon$ has the same definition as $\lambda_\epsilon^\perp$ (see \eqref{eq:alternative-transverse-def}) except $e_i$ is a basis of $H_I^\perp \cap \tilde v^\perp$, where $\tilde v = (0,0,1,0,0,1,\dots,0,0,1)$ (in other words, only look at $x$ and not $y$). Thus, to prove \eqref{eq:lorenz96-thm} it suffices to show it with $\lambda^\perp_\epsilon$ replaced by $\tilde \lambda_\epsilon$. Note that
$$\tilde \lambda_\epsilon = \max_{e_i} \lim_{t \to \infty}\frac{1}{t}\ln \|x^\epsilon(t)\| \,, \quad x^\epsilon(0) = e_i$$
(compare to \eqref{eq:alternative-transverse-def}) where the linearization of \eqref{eq:96-not-linearized} in the $H_I^\perp \cap \tilde v^\perp \cong \R^{2K-1}$ directions for initial conditions in $H_I \cong \R^K$ is given by
\begin{align*}
    dz^\epsilon(t)_j &= -\epsilon z^\epsilon(t)_jdt + \sqrt{\epsilon}dW_t^{3j} \\
    dx^\epsilon(t) &= [-\epsilon x^\epsilon(t) + (\Id - P)A(z^\epsilon(t))x^\epsilon(t)]dt \,.
\end{align*}
(In the setting of \Cref{sec:setup} we have $n = 2K - 1$ and $d = K$.)
If we can verify \Cref{as:nondegen} then \Cref{thm:main} implies
$$\lim_{\epsilon \downarrow 0} \tilde \lambda_\epsilon = \E[\lambda((\Id-P)A(Z))] \,,$$
which implies the equality of \eqref{eq:lorenz96-thm} because $\lambda((\Id-P)A(Z)) = 0 \vee \lambda(A(Z)) = \lambda(A(Z))$ ($v$ is an eigenvector of $A(Z)$ with eigenvalue $0$). Following \eqref{eq:96-az-def} and \Cref{def:matrix-notation}, we have $A = 0$ and  $B_k = (\Id-P)C_k$ where
\begin{equ}[e:defCk]
C_k x = x_{2k-2} e_{2k-1} + (x_{2k+1}-x_{2k-1})e_{2k} - x_{2k} e_{2k+1}\;,
\end{equ}
where $\{e_k\}$ is the standard basis for $\R^{2K}$ and all indices are taken modulo $2K$.
Then \Cref{as:nondegen} will be verified if we can show that $\{C_k \mid k = 0,1,\dots,K-1\}$ has no nontrivial invariant subspaces in $\R^{2K}$ besides $\R v$. Indeed, if $W \subset  (\Id - P)\R^{2K}$ is an invariant subspace for all $B_k$, then $W \oplus \R v$ is an invariant subspace for all $C_k$ because $C_k v = 0$ and for $w \in W$ we have $C_kw = B_kw + PC_k w \in W + \R v$.

By \eqref{e:defCk}, each of the $C_k$ maps basis vectors to basis vectors (modulo sign changes) and, 
for any $\ell, m$, one can find
$M \ge 0$ and indices $k_i$ such that $e_m = \pm C_{k_1}\cdots C_{k_M}e_\ell$. 
It therefore remains to show that if $x \in \R^{2K}$ is such that $x \not \in \R v$, then
one can find some polynomial expression of the $C$'s mapping $x$ to a basis vector.
To show this, it suffices to observe that
\begin{equ}
C_{k} C_{k+1} x = x_{2k} e_{2k}\;,\qquad 
 C_k^2 x = (x_{2k-1}- x_{2k+1}) e_{2k+1} - (x_{2k}+x_{2k-2}) e_{2k}\;.
\end{equ}
If $x$ has a non-zero even entry, then the first identity yields the desired property.
If not, then the second identity does, unless the odd entries are all equal, in which case
$x$ is proportional to $v$, but this was ruled out.
Thus, \Cref{as:nondegen} and the equality of \eqref{eq:lorenz96-thm} are verified.

    For the inequality of \eqref{eq:lorenz96-thm}, first note that $\operatorname{tr}(A(z)) = 0$ (all of the diagonal entries are $0$). Thus, the eigenvalues sum to $0$, implying $\lambda(A(z)) \geq 0$. By continuity of $\lambda$ in $z$ (see \Cref{lem:cty-of-eigenspaces} below) it suffices to find a single $z \in \R^K$ such that $\lambda(z) > 0$. One can check by hand using the characteristic polynomial that $z^* = (1,2,0,\dots,0)$ works (see \cite[Lem.~5.7]{transverse-lyap}). Indeed, $(5,5,-2,4,0,0,\dots,0)$ is an eigenvector of $A(z^*)$ with eigenvalue $1$.
\end{proof}

\subsection{Lorenz 63 System}
In \cite{lorenz}, the authors consider the Lorenz 63 system with additive forcing in one direction:
\begin{equation*}
    \begin{aligned}
        dX &= \sigma(Y-X)dt \\
    dY &= [X(\rho - Z) - Y]dt \\
    dZ &= [-\beta Z + XY]dt + \hat \alpha dW_t \,,
    \end{aligned}
\end{equation*}
where $\sigma,\rho,\beta,\hat \alpha > 0$ are constants with $\rho < 1$, $W_t$ is a standard one-dimensional Brownian motion, and $(X_t,Y_t,Z_t) \in \R^3$. In \cite[Theorem 1.1]{lorenz}, they prove a phase transition: when $\hat \alpha$ is small enough there is a unique invariant measure, but when $\hat \alpha$ is large enough there are multiple. As in \Cref{thm:navier-stokes} and our Lorenz 96 example, the uniqueness or nonuniqueness of the invariant measure is completely determined by the sign of a certain Lyapunov exponent (see \Cref{thm:persistence-extinction}).

After a change of variables detailed in \cite[Sec.~2]{lorenz}, the relevant quantity is the transverse Lyapunov exponent $\lambda_\alpha$ of the system
\begin{equation}\label{eq:lorenz}
    \begin{aligned}
        dx &= ydt \\
        dy &= [x(z-2)-2y]dt \\
        dz &= [-\gamma(z - z_*) - x(x+\eta y)]dt + \alpha dW_t \,,
    \end{aligned}
\end{equation}
where $\gamma, \eta > 0$, $z_* \leq 2$, and $\alpha$ is some multiple of $\hat \alpha$.
\begin{rem}
    This change of variables is only done for consistency with \cite{lorenz} where it leads to
    nicer expressions.
\end{rem}

Formally,
$$\lambda_\alpha \coloneqq \lim_{t \to \infty} \frac{1}{t}\ln \sqrt{\hat x_t^2 + \hat y_t^2} \,,$$
where
\begin{equation*}
        d\hat x = \hat ydt \;,\quad
        d\hat y = [\hat x(\hat z-2)-2\hat y]dt \;,\quad
        d\hat z = -\gamma(\hat z - z_*)dt + \alpha dW_t \;.
\end{equation*}
It follows from \cite[Props~3.2--3.3]{lorenz} and a standard argument (see for example \cite[Section 4.2]{extinction}) that $\lambda_\alpha$ is well-defined: the limit exists almost surely for any initial condition with $(\hat x(0),\hat y(0)) \neq (0,0)$ and the value is independent of the initial condition.
Alternatively, $\lambda_\alpha$ can be defined as the $\Lambda^-$ or $\Lambda^+$ in \Cref{thm:persistence-extinction} (here $P_{\invm}$ consists of only one measure). Or one could cite the multiplicative ergodic theorem.

The asymptotics of $\lambda_\alpha$ as $\alpha \to 0$ or $\infty$ are given in \cite[Thm~5.2]{lorenz}, and we note that the $\alpha \to 0$ case is also an immediate consequence of \cite[Thm~3.11]{extinction} (see the discussion in \cite[Sec.~10.2]{extinction}). Below we use \Cref{thm:main} to give an alternative proof of the $\alpha \to \infty$ case:
\begin{thm}\label{thm:lorenz}
      $$\lim_{\alpha \to \infty} \alpha^{-1/2}\lambda_\alpha = \frac{\Gamma(\frac{3}{4})}{2\gamma^{1/4}\pi^{1/2}} \,,$$
    where $\Gamma(z) = \int_0^\infty t^{z-1}e^{-t}dt$ is the gamma function.
\end{thm}

\begin{proof}
    We start with a change of variables
    $$(\tilde x_t,\tilde y_t,\tilde z_t) \coloneqq (\alpha^{1/2}\hat x_{\alpha^{-1/2}t},\hat y_{\alpha^{-1/2}t},\alpha^{-1}(\hat z_{\alpha^{-1/2}t} - z_*))\;,$$
    and we set $\epsilon \coloneqq \alpha^{-1/2}$, so that
    \begin{equation}\label{eq:lorenz-changed}
        \begin{aligned}
             d\tilde x &= \tilde ydt \\
        d\tilde y &= [\tilde x \tilde z + (z_*-2)\epsilon^{2}\tilde x-2\epsilon\tilde y]dt \\
        d\tilde z &= -\epsilon \gamma\tilde zdt + \epsilon^{1/2} dW_t \,.
        \end{aligned}
    \end{equation}

    Thus, in the notation of \Cref{def:matrix-notation} ($n = 2$, $d = 1$) we have
    \begin{equs}
        A &\coloneqq \mat{0}{1}{0}{0} \quad&\quad
        B_1 &\coloneqq \mat{0}{0}{1}{0} \\ {}
        [A,B_1] &= \mat{-1}{0}{0}{1} \quad&\quad
        \operatorname{ad}_A^2(B_1) &= \mat{0}{-2}{0}{0}\,,
    \end{equs}
    which satisfies \Cref{as:nondegen}. Indeed, given any nonzero vector $v = (a,b) \in \R^2$ we have $B_1v = (0,a)$, $\operatorname{ad}_A^2(B_1)v = (-2b,0)$, $\operatorname{ad}_A^2(B_1)B_1v = (-2a,0)$, and $B_1\operatorname{ad}_A^2(B_1)v = (0,-2b)$, which together necessarily span $\R^2$. Note that
    $$\lambda^\epsilon \coloneqq \lim_{t \to \infty} \frac{1}{t}\ln \|(\tilde x, \tilde y)\|$$
    is related to $\lambda_\alpha$ via
    \begin{equation}\label{eq:alpha-epsilon-relation}
        \lambda_\alpha = \alpha^{1/2}\lambda^{\epsilon} \,.
    \end{equation}
    Since $\sigma(A + zB_1) = \{\pm \sqrt{z}\}$, we have $\lambda(A + zB_1) = \sqrt{z}1_{z > 0}$ (see \Cref{def:sigma-lambda}) and \Cref{thm:main} yields
    $$\lim_{\epsilon \to 0} \lambda^\epsilon = \frac{\gamma^{1/2}}{\pi^{1/2}}\int_0^\infty \sqrt{z}e^{-\gamma z^2}dz = \frac{1}{2\gamma^{1/4}\pi^{1/2}}\int_0^\infty u^{-1/4}e^{-u}du = \frac{\Gamma(\frac{3}{4})}{2\gamma^{1/4}\pi^{1/2}} \,.$$
    The theorem then follows from \eqref{eq:alpha-epsilon-relation}.
\end{proof}

\section{Proof of Main Theorem}
\label{sec:main}

After some preliminary statements we will present a proof of \Cref{thm:main}. For this section, recall \Cref{def:sigma-lambda}. The generalized eigenspaces $V_1,\dots,V_k$ of a linear operator $T$ on $\C^n$ are given by
$$V_i \coloneqq \bigcup_{j=1}^\infty \operatorname{ker}(T - \lambda_iI)^j\;,$$
where $\sigma(T) = \{\lambda_1,\dots,\lambda_k\}$. Note that $\bigoplus_{i=1}^k V_i = \C^n$. However, if $T$ is a real matrix then $T\bar v = \overline{Tv}$ for $v \in \C^n$, so the $V_i$'s come in conjugate pairs and thus the projections defined below will be well-defined maps from $\R^n$ to $\R^n$:

\begin{deff}\label{def:proj-eta}
    Let $z \in \R^d$ and $\eta \ge 0$. Write $\sigma(A(z)) = \{\lambda_1,\dots,\lambda_k\}$, let $V_1,\dots,V_k$ be the corresponding generalized eigenspaces, and suppose $w \in \R^n$ satisfies $w = \sum_{i=1}^k v_i$ for $v_i \in V_i$. Then
    $$P_\eta(z)w \coloneqq \sum_{\{i \mid \Re \lambda_i \geq \lambda(A(z)) - \eta\}} v_i \,.$$
\end{deff}

We denote by $P^\epsilon_{\invm}$ the set of all invariant measures of the Markov process 
$(z^\epsilon(t), v^\epsilon(t))$
given by \eqref{eq:zvr-system-def}, and when $\epsilon = 0$ we simply write $P_{\invm}$.

\begin{lem}\label{lem:description-of-Pinv}
    Let $\mu \in P_{\invm}$ be ergodic. Then there is some $z_0$ such that $\mu(\{z = z_0\}) = 1$. Furthermore, there is some $i$ such that $\mu(\{v \in \operatorname{Span}(\{V_j \mid \Re \lambda_j = \Re \lambda_i\})\}) = 1$, where $V_1,\dots,V_k$ are the generalized eigenspaces of $A(z_0)$ corresponding to the eigenvalues $\lambda_1,\dots,\lambda_k$. Additionally, we have
$$\int \Lambda(z,v)d\mu(z,v) = \Re \lambda_i \,.$$
In particular, if $\mu(\{P_0(z)v = 0\}) = 0$ then
$$\int \Lambda(z,v)d\mu(z,v) = \lambda(A(z_0)) \,.$$
\end{lem}
\begin{proof}
    Since $z^0(t)$ is constant, it is clear that ergodic $\mu$ must have $\mu(\{z = z_0\}) = 1$ for some $z_0$.

Then the proof follows by an elementary computation on Jordan blocks of $A(z_0)$. First, we show that for any (real, nonzero) $v \in \operatorname{Span}(V_i,\overline V_i)$ we have
    \begin{equation}\label{eq:eigenvalue-equals-growth-rate}
        \lim_{t \to \infty} \frac{1}{t}\ln\|e^{tA(z_0)} v\| = \Re \lambda_i \,.
    \end{equation}
    Indeed, by the Jordan decomposition theorem, the restriction of $A(z_0)$ to $V_i$ is given by 
    $\lambda_i + N$, for some nilpotent $N$.
    Write $v = w+\bar w$ with $w \in V_i$. If $\lambda_i$ is real then we have $w = v/2$, otherwise we choose a norm on $\C^n$ for which $V_i$ and $\bar V_i$ are orthogonal (since all norms are equivalent, a change of norm will not affect the left-hand side \eqref{eq:eigenvalue-equals-growth-rate}). Thus, for $c = 2$ or $c = 4$ we have
    $$\|e^{tA(z_0)}v\|^2 = c\|e^{\lambda_i t}e^{Nt}w\|^2 = ce^{2\Re \lambda_i t}\|e^{Nt}w\|^2\;.$$
    
    Writing $k = \sup\{j \ge 0\,:\, N^j w \neq 0\}$, we have $k < \infty$ by the nilpotency of $N$, so that
    \begin{equ}
    e^{Nt}w = \sum_{j \le k} \frac{t^j}{j!} N^j w\qquad\text{therefore}\qquad
    \|e^{Nt}w\| = \frac{t^k}{k!}\|N^k w\|\bigl(1+\CO(t^{-1})\bigr)\;,
    \end{equ}
and \eqref{eq:eigenvalue-equals-growth-rate} follows.

   A general $v$ can be written as $v = \sum_j (w_j + \bar w_j)$ with $w_j \in V_j$. 
Choosing $i$ such that $w_i \neq 0$ and, for all $j$ such that $w_j \neq 0$ one has $\Re \lambda_j \le \Re \lambda_i$,
it follows from \eqref{eq:eigenvalue-equals-growth-rate} that
   $$  \lim_{t \to \infty} \frac{1}{t}\ln\|e^{tA(z_0)} v\| = \max_{j \,:\, w_j \neq 0}\Re \lambda_{j} = \Re \lambda_i \,,$$
   and that $v^0(t) = e^{tA(z_0)} v/\|e^{tA(z_0)} v\|$ approaches $\operatorname{Span}(\{V_j \mid \Re \lambda_j = \Re \lambda_i\})$ (the other components grow at a slower rate). The claim then follows from \eqref{eq:lyap-exp-equals-ergodic-average}. 
\end{proof}

\begin{deff}\label{def:eta-star}
    For a matrix $T \in \R^{n \times n}$ we define the spectral gap
    $$\eta^*(T) \coloneqq \lambda(T) - \sup\{\Re \lambda \mid \lambda \in \sigma(T), \Re\lambda \neq \lambda(T)\} \,,$$
    and for $z \in \R^d$ we set $\eta^*(z) \coloneqq \eta^*(A(z))$.
\end{deff}

\begin{rem}
If all values in the spectrum of a matrix $T$ have the same real part, then we set $\eta^*(T) = +\infty$,
consistent with the convention $\sup\emptyset = -\infty$.
\end{rem}

In \Cref{sec:mass-moving} we will show the following two lemmas. The first says that with high probability the time it takes to have $\|P_\eta(z)v\| \geq \delta$ is logarithmic in $\epsilon^{-1}$. The second says that with high probability it takes polynomial in $\epsilon^{-1}$ time to have $\|P_\eta(z)v\|$ significantly less than $\delta$. As $\eta$ drops below $\eta^*(z)$, we can then conclude that on average much more time is spent away from the set $P_0(z)v = 0$ than near it, and by \Cref{lem:description-of-Pinv} this implies that our growth rate is close to $\lambda(A(z))$.

\begin{lem}\label{lem:mass-from-high-to-low}
    For every $(\tilde z_0,\tilde v_0) \in \R^d \times S^{n-1}$ and $\eta \in (0, \eta^*(\tilde z_0))$ there is some $\delta, C > 0$ such that for all $(z_0,v_0)$ with $$\|\tilde z_0 - z_0\| + \|\tilde v_0 - v_0\| \leq \delta$$
    and $\epsilon < \delta$ we have
    $$\Prb\Big(\sup_{t \leq C\ln\epsilon^{-1}} \|P_\eta(z_0)v^\epsilon(t)\| \geq \delta \Big) \geq \delta$$
    (where the process $(z^\epsilon(t), v^\epsilon(t))$ is defined as in \eqref{eq:zv-def} with initial condition $(z_0,v_0)$).
\end{lem}

The proof of this lemma is postponed to Section~\ref{sec:firstLemma}.

\begin{lem}\label{lem:mass-stays-on-high}
For all $\tilde z_0 \in \R^d$, $\delta \in (0,\frac16]$, and $\eta \in (0, \eta^*(\tilde z_0))$ there are some $\rho \in (0,1)$ and $T > 0$ such that for all $(z_0,v_0)$ with
    $$\|\tilde z_0 - z_0\| \leq \rho \quad \text{and} \quad \|P_\eta(z_0)v_0\| \geq \delta$$
    and $\epsilon < \rho$ we have
    $$\Prb\Big(\inf_{T \leq t \leq \epsilon^{-1/2}} \|P_{\eta}(z_0)v^\epsilon(t)\| \leq \delta\Big) \leq \delta$$
    (where the process $(z^\epsilon(t), v^\epsilon(t))$ is defined as in \eqref{eq:zv-def} with initial condition $(z_0,v_0)$).
\end{lem}

The proof of this lemma is postponed to Section~\ref{sec:secondLemma}.
Given the two lemmas above, we carry out the following construction:

\begin{lem}\label{lem:construction-of-mn}
    There exist open sets $\mcM_N \subset \R^d \times S^{n-1}$ for $N \in \N$ satisfying
\begin{equation}\label{eq:mn-contains-p0}
    \mcM_\infty \coloneqq \{(z,v) \in \R^d \times S^{n-1} \mid P_0(z)v = 0\} \subset \bigcup_{K = 1}^\infty \bigcap_{N=K}^\infty \mcM_N
\end{equation}
and
\begin{equation}\label{eq:mn-goes-to-0}
    \lim_{N \to \infty} \limsup_{\epsilon \downarrow 0} \sup_{\mu^\epsilon \in P^\epsilon_{\invm}} \mu^\epsilon(\mcM_N) = 0\,.
\end{equation}
\end{lem}

\begin{proof}
Throughout the proof, $(z_0,v_0)$ denotes the initial condition for the process defined in \eqref{eq:zvr-system-def} and \Cref{lem:well-posed} ($r_0$ is not relevant). Fix $N \in \N$ and let $\K_N \coloneqq \{z \in \R^d \mid \|z\| \leq N\}$.

Let $z \in \K_N$ and set $\eta(z) = \frac{1}{2}\eta^*(z) \wedge N^{-1}$. By \Cref{lem:mass-from-high-to-low} and compactness of $S^{n-1}$ there are some $0 < \rho(z) < \delta(z) < N^{-1}$, $C(z) > 0$ such that for all $\epsilon < \rho(z)$
    \begin{equation}\label{eq:first-attempt}
        \|z_0 - z\| < \rho(z) \implies \Prb\Big(\sup_{t \leq C(z)\ln\epsilon^{-1}} \|P_{\eta(z)}(z_0)v^\epsilon(t)\| \geq 2\delta(z) \Big) \geq 2\delta(z) \,,
    \end{equation}
    uniformly over $v_0 \in S^{n-1}$.
    After decreasing $\rho(z)$, by \Cref{lem:mass-stays-on-high} there are $T(z) > 0$ such that
    \begin{equation}\label{eq:after-successfull-attempt}
    \begin{aligned}
         &\|z_0 - z\| < \rho(z) \text{ and } \|P_{\eta(z)}(z_0)v_0\| \geq \delta(z)\\
          &\implies \Prb\Big(\inf_{T(z) \leq t \leq \epsilon^{-1/2}} \|P_{\eta(z)}(z_0)v^\epsilon(t)\| \leq \delta(z)\Big) \leq \delta(z)\,.
    \end{aligned}
    \end{equation}
    After possibly decreasing $\rho(z)$ again, we have by \Cref{lem:cty-of-eigenspaces} that
    \begin{equation}\label{eq:proj-are-close}
        \|z_0 - z\| < \rho(z) \implies \|P_{\eta(z)}(z_0) - P_{\eta(z)}(z)\| \leq \frac{\delta(z)}{4}
    \end{equation}
    and $z_0 \mapsto P_{\eta(z)}(z_0)$ is continuous on the set where $\|z_0 - z\| < \rho(z)$.
    By compactness of $\K_N$, there exist $k \in \N$ (depending on $N$) and $z_1,\dots,z_k \in \K_N$ such that for all $z_0 \in \K_N$ there exists $i \in [k]$ with $\|z_0 - z_i\| < \frac{1}{4}\rho(z_i)$. We define
    \begin{equ}\label{eq:def-of-mn}
        \mcM_N \coloneqq \Big\{(z_0,v_0) \Bigm| \|P_{\eta(z_i)}(z_0)v_0\| < \frac{\delta(z_i)}{2} \text{ whenever } \|z_0 - z_i\| \le \frac{\rho(z_i)}2 \Big\} \,,
    \end{equ}
    which, being a finite intersection of open sets by the above, is indeed open.
    Then \eqref{eq:mn-contains-p0} holds because if $P_0(z)v = 0$ and we take
    $N > \|z\| \vee \eta^*(z)^{-1}$ then whenever $\|z - z_i\| \leq \rho(z_i)/2$ one has $P_0(z) = P_{\eta(z_i)}(z)$ (because $\eta(z_i) \leq N^{-1} < \eta^*(z)$), so that $\|P_{\eta(z_i)}(z)v\| = 0$, which shows that $(z,v) \in \mcM_N$ for all large enough $N$.

    It remains to show \eqref{eq:mn-goes-to-0}. Fix $N \in \N$ and $z_0 \in \K_N$, and choose $i$ as in \eqref{eq:def-of-mn}. In what follows, to ease notation we set $\rho = \rho(z_i)$, $\delta = \delta(z_i)$, etc. We also assume that $\epsilon^{1/8} < \frac{1}{4}\min_{j \in [k]} \rho(z_j)$.
    
    Note that by \Cref{lem:z-remains-small}
    \begin{equation}\label{eq:upper-bound-on-z}
        \lim_{\epsilon \downarrow 0}\sup_{z_0 \in \K_{N+1}}\Prb\Big(\sup_{t \leq \epsilon^{-2/3}} \|z^\epsilon(t) - z_0\| > 
\epsilon^{1/8} \Big) = 0 \,,
    \end{equation}
    so by possibly reducing $\epsilon$ we may assume that
    $$\sup_{z_0 \in \K_{N+1}}\Prb\Big(\sup_{t \leq \epsilon^{-2/3}} \|z^\epsilon(t) - z_0\| > 
\epsilon^{1/8} \Big) \leq \frac{\min_{j \in [k]} \delta(z_j)}{3} \leq \frac{\delta}{3} \,.$$
In particular, with $i$ as in \eqref{eq:def-of-mn},
\begin{equation}\label{eq:def-of-e1}
    \Prb(E_1) \coloneqq \Prb\Big(\sup_{t \leq \epsilon^{-2/3} }\|z^\epsilon(t) - z_i\| < \frac{\rho}{2} \Big) \geq 1 - \frac{\delta}{3}\,.
\end{equation}
Let
$$\tau \coloneqq \inf\{t \geq 0 \mid \|P_\eta(z_0)v^\epsilon(t)\| \geq 3\delta/2\}$$
and $\tau_1 = \tau \wedge \inf \{t \geq 0\mid \|z^\epsilon(t)-z_i\|\ge \rho/2\}$.
By \eqref{eq:first-attempt} and the Markov property (for the third inequality below) and \eqref{eq:def-of-e1}, we have for all $m \leq \epsilon^{-2/3}(C\ln \epsilon^{-1})^{-1}$ that
\begin{align*}
    \Prb\Big(\tau > mC\ln \epsilon^{-1} &\text{ or } E_1^C\Big) \leq \Prb\Big(\tau_1 > mC\ln \epsilon^{-1}\Big) + \Prb(E_1^c) \\
    &\leq \E\Big[1_{\tau_1 > (m-1)C\ln\epsilon^{-1}}\Prb\Big(\tau > mC\ln \epsilon^{-1} \,\Big|\, \F_{(m-1)C\ln \epsilon^{-1}}\Big)\Big] + \Prb(E_1^c)\\
    &\leq (1 - 2\delta)\Prb\Big(\tau_1 > (m-1)C\ln \epsilon^{-1}\Big) + \Prb(E_1^c) \\
    &\leq \dots \leq (1 - 2\delta)^m + \frac{\delta}{3} \,.
\end{align*}
Choose $M$ such that $(1-2\delta)^M < 2\delta/3$, so that the above becomes
$$\Prb(E_2) \coloneqq \Prb\Big(\tau \leq MC\ln \epsilon^{-1} \text{ and } E_1 \Big) \geq 1 - \delta \,.$$
On the event $E_2$, by \eqref{eq:proj-are-close} we have $\|P_\eta(z^\epsilon(\tau))v^\epsilon(\tau)\| \geq 3\delta/2 - \delta/2 = \delta$, so we may apply \eqref{eq:after-successfull-attempt} and the strong Markov property to conclude
$$\Prb\Big(\inf_{MC\ln \epsilon^{-1}+T \leq t \leq \epsilon^{-1/2}} \|P_{\eta}(z^\epsilon(\tau))v^\epsilon(t)\| \geq \delta \text{ and } E_2\Big) \geq 1 - 2\delta \,.$$
By $E_2 \subset E_1$ and \eqref{eq:proj-are-close},
\begin{equ}[e:lowerBound]
\Prb\Big(\inf_{MC\ln \epsilon^{-1}+T \leq t \leq \epsilon^{-1/2}} \|P_{\eta}(z^\epsilon(t))v^\epsilon(t)\| \geq \frac{\delta}{2} \text{ and } E_1\Big) \geq 1 - 2\delta \,.
\end{equ}

To get uniform bounds over $z_0$, set $\bar\delta \coloneqq \max_{i=1,\dots,k} \delta(z_i)$, 
$\bar T \coloneqq \max_{i=1,\dots,k} T(z_i)$, and similarly for $\bar M$ and $\bar C$.
Combining \eqref{e:lowerBound} with the definition \eqref{eq:def-of-mn}, we have
$$\E\Big[\frac{1}{\epsilon^{-1/2}}\int_0^{\epsilon^{-1/2}} 1_{\mcM_N^c}(z^\epsilon(t),v^\epsilon(t))dt\Big] \geq (1-2\bar\delta)\Big(1 - \frac{\bar M \bar C\ln \epsilon^{-1}+\bar T}{\epsilon^{-1/2}}\Big)1_{z_0 \in \K_N} \,.$$
Given any $\mu^\epsilon \in P^\epsilon_{\invm}$ and integrating both sides of this bound over $(z_0,v_0)$
against $\mu^\epsilon$ yields 
$$\mu^\epsilon(\mcM_N^c) \geq (1-2\bar\delta)\Big(1 - \frac{\bar M\bar C\ln \epsilon^{-1}+\bar T}{\epsilon^{-1/2}}\Big)\mu^\epsilon(\K_{N-1}) \,.$$
Since $\mu^\epsilon(\K_{N}) = \Prb(Z \in \K_{N})$ where $Z$ is the invariant measure for the OU 
process \eqref{eq:z-def} as in \Cref{thm:main}, we conclude that
$$\lim_{\epsilon \downarrow 0} \sup_{\mu^\epsilon \in P_{\invm}^\epsilon} \mu^\epsilon(\mcM_N) \leq 1 - (1-2\bar\delta)\Prb(Z \in \K_{N-1}) \,.$$
Taking $N \to \infty$ forces $\bar\delta \to 0$ (we chose $\delta(z) < N^{-1}$) and $\Prb(Z \in \K_{N-1}) \to 1$, so we conclude \eqref{eq:mn-goes-to-0}.
\end{proof}

With all of these ingredients in place, in particular \Cref{lem:construction-of-mn}, 
the proof of \Cref{thm:main} is now relatively straightforward.

\begin{proof}
    First note that the process $(z^\epsilon(t), v^\epsilon(t), \epsilon(t))$ where $\epsilon(t) \equiv \epsilon$ is a Feller Markov process on $\R^d \times S^{n-1} \times [0,1]$. Indeed, the coefficients of the $(z,v)$ portion of the system \eqref{eq:zvr-system-def} along with $d\epsilon(t) = 0$ are smooth and have linear growth (recall that $v$ is restricted to lie on the unit sphere), so it is a standard fact that one has continuous dependence on initial conditions (see for example \cite[Theorem 8.7]{LeGall}). It follows that if $\mu^\epsilon \in P_{\invm}^\epsilon$ and $\mu^\epsilon \to \mu$ along a subsequence of $\epsilon \downarrow 0$, then $\mu \in P_{\invm}$ (limits of invariant measures are still invariant measures if the process is Feller). 
    
 Since every $\mu^\epsilon$ satisfies $\mu^\epsilon(\{z \in A\}) = \Prb(Z \in A)$
  for all Borel subsets $A$ of $\R^d$ (where $Z$ has the law of the invariant distribution of the OU process as defined in \Cref{thm:main}), the family $\{\mu^\epsilon \mid \epsilon \in (0,1], \mu^\epsilon \in P^\epsilon_{\invm}\}$ is tight and its limit points as $\epsilon \downarrow 0$ lie in $P_{\invm}$. If $\mu$ is such a limit point, then by using Portmanteau on \eqref{eq:mn-goes-to-0} (recall from \Cref{lem:construction-of-mn} that $\mcM_N$ are open) and then \eqref{eq:mn-contains-p0} we conclude that
    $$\mu(\{(z,v) \in \R^d \times S^{n-1} \mid P_0(z)v = 0\}) = 0 \,.$$
    By \Cref{lem:description-of-Pinv} and ergodic decomposition, this implies that
    \begin{equation}\label{eq:lambda-star-is-correct}
        \int \Lambda(z,v)d\mu(z,v) = \E[\lambda(A(Z))] = \lambda^* \,.
    \end{equation}
    Since $|\Lambda^\epsilon(z,v)| \lesssim 1+|z|$ and $\mu^\epsilon(|z|^2) = \E[|Z|^2] < \infty$ uniformly over
    all $\epsilon$ and $\mu^\epsilon \in P^\epsilon_{\invm}$, by \cite[Proposition 4.15]{BenaimHurth22} and \eqref{eq:lambda-star-is-correct} we deduce \eqref{eq:main-second}.

To show \eqref{eq:main-first}, we first claim that almost surely
$$\lambda_\epsilon \in S^\epsilon \coloneqq \Big\{\int \Lambda^\epsilon(z,v)d\mu(z,v) \Bigm| \mu \in P^\epsilon_{\invm}\Big\} \,.$$
To show this, note that \eqref{eq:dlnr} and \eqref{eq:lyap-exp-def} show that $\lambda_\epsilon = \lim_{t \to \infty} \mu_t(\Lambda^\epsilon)$, where $\mu_t$ denotes the normalised empirical occupation measure for the process
$(z^\epsilon,v^\epsilon)$. The fact that the $\mu_t$ are tight (as a family indexed by $t$, for any fixed $\epsilon$) and that their limit points lie in $P^\epsilon_{\invm}$ was shown for example
in \cite[Thm~2.2]{persistence} (see also \cite[Thm~4.20]{BenaimHurth22}).

Given $\epsilon_n \downarrow 0$, we can find random variables $\mu_n \in P^{\epsilon_n}_{\invm}$ such that, almost surely,
$$\lambda_{\epsilon_n} = \int \Lambda^{\epsilon_n}(z,v)d\mu_n(z,v) \,.$$
The tightness of the collection of all invariant measures implies that, for every subsequence of $\mu_n$, there is a further subsequence which converges to some $\mu \in P_{\invm}$, and that this $\mu$ satisfies \eqref{eq:lambda-star-is-correct}. Then by continuity of $\Lambda^{\epsilon}$ in $\epsilon$, $|\Lambda^\epsilon(z,v)| \lesssim 1+|z|$, and $\mu_n(\{z \in \cdot\}) = \Prb(Z \in \cdot)$, we conclude that
$$\lim_{n\to\infty}\lambda_{\epsilon_n} = \lim_{n \to \infty} \int \Lambda^{\epsilon_n}(z,v)d\mu_n(z,v) = \int \Lambda(z,v)d\mu(z,v) = \lambda^*$$
(the existence of the limits above are implied because for all subsequences there is a subsubsequence which satisfies the limit).
\end{proof}

\section{Mass Generation and Preservation on Low Modes}\label{sec:mass-moving}
In this section we present proofs of \Cref{lem:mass-from-high-to-low,lem:mass-stays-on-high}. The proofs use the chaos expansion of \eqref{eq:x-def} with respect to a suitable Gaussian process. In particular, we make the following definition:
\begin{deff}\label{def:chaos-expansion}
    Given $z_0 \in \R^d$ and $\epsilon \in [0,1]$, let $\tilde z^\epsilon(t) \coloneqq z^\epsilon(t) - z_0$, where $z^\epsilon(t)$ is as in \eqref{eq:zv-def} with initial condition $z^\epsilon(0)=z_0$. For a function $f: \Delta^k(t) \to \R$, where
    $$\Delta^k(t) \coloneqq \{(t_1,\dots,t_k) \mid t \geq t_1 \geq \dots \geq t_k \geq 0\} \,,$$
    and a multi-index $J \in [d]^k$, we define
    $$I_J(f) \coloneqq \int_{\Delta^k(t)} f(t_1,\dots,t_k)\prod_{j=1}^k \tilde z_{J_j}(t_j)\,dt_k\cdots dt_1 \,,$$
    and we also overload this notation to apply to $\R^n$-valued $f$.
    When $k = 0$ this should be interpreted as $I_0(f) = f$, where $f$ is a constant.
\end{deff}

As in \cite{Tommaso,Batchelor}, we will approximate the solution to \eqref{eq:x-def} by an element in a fixed Wiener chaos
and use lower bounds on that to show that it necessarily leaves any ``bad'' region relatively quickly. 
Setting $\tilde A = A^\epsilon(z_0)$ and iterating the mild formulation of \eqref{eq:x-def}, we obtain for all $K \in \N_0$ that
    \begin{equation}\label{eq:chaos-expansion}
        \begin{aligned}
              x^\epsilon(t) &= \sum_{k=0}^K \sum_{J \in [d]^k} I_J\Big(e^{t\tilde A}\Big[\prod_{j=1}^k e^{-t_j\tilde A}B_{J_j}e^{t_j\tilde A}\Big]x_0\Big) \\
              &\quad + \sum_{J \in [d]^{K+1}} I_J\Big(e^{t\tilde A}\Big[\prod_{j=1}^{K+1} e^{-t_j\tilde A}B_{J_j}e^{t_j\tilde A}\Big]e^{-t_{K+1}\tilde A}x^\epsilon({t_{K+1}})\Big) \\
              &\eqqcolon X_K^\epsilon(t) + R_K^\epsilon(t)
        \end{aligned}
    \end{equation}
    (where $\prod_{j=1}^k a_j \coloneqq a_1\dots a_k$).

\subsection{Control over \texorpdfstring{$X_K$}{XK} and \texorpdfstring{$R_K$}{RK}}
We analyze the terms $X^\epsilon_K(t)$ and $R^\epsilon_K(t)$ separately, first showing that $X^\epsilon_K(t)$ has some chance of being at least (about) $\epsilon^{K/2}e^{\lambda t}$, and then that $R^\epsilon_K(t)$ should be at most $\epsilon^{(K+1)/2}e^{\lambda t}$, where $\lambda$ is the growth rate associated to $e^{t \tilde A}$. For the lemma below, recall \Cref{lem:well-posed} and in particular note that $x_0 = \|x_0\|v_0$.

\begin{lem}\label{lem:mass-generation}
    For every $(\tilde z_0,\tilde v_0) \in \R^d \times S^{n-1}$ there is some $K \in \N_0$ such that for all $\eta > 0$ there are some $\delta, C > 0$ such that for all $(z_0,v_0)$ with $$\|\tilde z_0 - z_0\| + \|\tilde v_0 - v_0\| \leq \delta \,,$$
$\epsilon < \delta$, and $t \geq C$ we have
    $$\Prb(\|X_K^\epsilon(t)\| \geq  \epsilon^{K/2}e^{(\lambda - \eta/2)t}\|x_0\|) \geq \delta \,,$$
    where $\lambda \coloneqq \lambda(A^\epsilon(z_0))$ (recall the notation $X^\epsilon_K(t)$ from \eqref{eq:chaos-expansion}).
\end{lem}

\begin{proof}
We may assume $\eta \in (0,\eta^*(\tilde z_0))$ (see \Cref{def:eta-star}). By \Cref{lem:algebra-prop}, there is some $K \in \N_0$, $J^* \in [d]^K$, and $k = (k_1,\dots,k_K) \in \N_0^K$ such that
     $$\Big\|P_{\eta}(\tilde z_0)\prod_{j=1}^K \operatorname{ad}_{A(\tilde z_0)}^{k_j}(B_{J^*_j})\tilde v_0\Big\| > 0 \,.$$
     Using the notation $\tilde A \coloneqq A^\epsilon(z_0)$, we conclude by \Cref{lem:cty-of-eigenspaces} that there 
     are constants $\delta, \rho > 0$ such that for all $(z_0,v_0)$ with
     $$\|\tilde z_0 - z_0\| + \|\tilde v_0 - v_0\| \leq \rho$$
     and $\epsilon < \rho$ we have
     $$\Big\|P_\eta(z_0)\prod_{j=1}^K \operatorname{ad}_{\tilde A}^{k_j}(B_{J^*_j})v_0\Big\| \geq \delta \,.$$
     Note also that $\|\prod_{j=1}^K \operatorname{ad}_{\tilde A}^{k_j}(B_{J^*_j})v_0\|$ is bounded above and below by positive constants. Then we may apply \Cref{lem:uniform-exp-lower-bound} to conclude that, for any fixed $c$ (to be chosen later),
     one can choose $t$ large enough and possibly lower the
     value of $\rho$ such that
     \begin{equation}\label{eq:the-term-that-gives-lower-bound}
         \Big\|P_\eta(z_0)e^{t\tilde A}\prod_{j=1}^K \operatorname{ad}_{\tilde A}^{k_j}(B_{J^*_j})v_0\Big\| \geq e^{(\lambda - c/3)t} \Big\|\prod_{j=1}^K \operatorname{ad}_{\tilde A}^{k_j}(B_{J^*_j})v_0\Big\| \geq e^{(\lambda - c/2)t} \,.
     \end{equation}
     (Recall that $\lambda \coloneqq \lambda(A^\epsilon(z_0))$.)
    For every $J \in [d]^\ell$ with $0 \leq \ell \leq K$, define
    $$f_J(t_1,\dots,t_\ell) \coloneqq e^{t \tilde A}\prod_{j=1}^\ell e^{-t_j\tilde A}B_{J_j}e^{t_j\tilde A}v_0 \,,$$
    so that by \eqref{eq:chaos-expansion} we have
    $$X_K^\epsilon(t) = \|x_0\|\sum_{\ell=0}^K \sum_{J \in [d]^\ell} I_J(f_J) \,.$$
    Given multiindices $I \in \N_0^\ell,J \in [d]^\ell$ it follows that
    \begin{equation}\label{eq:partial-I-f-J}
        \partial^I f_J(t_1,\dots,t_\ell) = (-1)^{|I|}e^{t \tilde A}\prod_{j=1}^\ell e^{-t_j\tilde A}\operatorname{ad}_{\tilde A}^{I_j}(B_{J_j})e^{t_j \tilde A}v_0 
    \end{equation}
    (where $\partial^I$ denotes $\prod_{j=1}^\ell\partial_{t_j}^{I_j}$).
Then by \eqref{eq:the-term-that-gives-lower-bound}  and $\|P_\eta(z_0)\|$ being uniformly bounded, (see \Cref{lem:cty-of-eigenspaces}, in particular $\|P_\eta(z_0)\| \le e^{ct/2}$ for $t$ large enough)
    \begin{equation}\label{eq:new-term-that-gives-lower-bound}
        \|\partial^{k} f_{J^*}(0)\| = \Big\|e^{t\tilde A}\prod_{j=1}^K \operatorname{ad}_{\tilde A}^{k_j}(B_{J^*_j})v_0\Big\| \geq e^{(\lambda - c)t}
    \end{equation}
    for large enough $t$ (recall $k \in \N_0^K$).
    
    On the other hand, it follows from \eqref{eq:partial-I-f-J} and \Cref{lem:uniform-exponential-bound} that for all $r$ there is some $M \geq 1$ such that
    \begin{equ}[e:boundDerf]
    \sup_{\Delta^K(t)} \|D^r f_{J^*}(t_1,\dots,t_K)\| \leq Me^{(\lambda + c)t}\;.
    \end{equ}
    When $K \geq 1$, using the above and \eqref{eq:new-term-that-gives-lower-bound}, by \Cref{lem:taylor-expansion} (with $m = k_1 + \dots k_K$) we conclude
$$\|f_{J^*}\|_{L^2(\Delta^K(1))} \geq e^{(\lambda-c)t}C(K,m)M^{-K/2-m}e^{-2ct(K/2+m)} \,.$$
When $K = 0$ then the above (with $C(K,m)$ removed) is a direct consequence of \eqref{eq:new-term-that-gives-lower-bound} because $f_{J^*}$ is a constant when $J^*$ is empty.
Using \eqref{e:boundDerf} (potentially enlarging $M$) and the notation from \Cref{lem:chaos-terms-l2-lower-bound} we obtain
$$\|f_{J^*}\|_{H^{2K}(\Delta^K(t))} \leq Me^{(\lambda + c)t} \,.$$
Then absorbing all constants into $a >0$, by \Cref{lem:chaos-terms-l2-lower-bound} we have
 $$\Prb\Big(\|X_K^\epsilon(t)\| \geq a\epsilon^{K/2}t^{-K}e^{2(\lambda-c)t}e^{-2ct(K+2m)}e^{-(\lambda + c)t}\|x_0\|\Big) \geq \frac{3^{-2K}}{4} \,,$$
so the claim follows by first choosing $c$ so that $c(2+2(K+2m)+1) < \eta/2$, then choosing $C \geq 1$ large enough so that $e^{[\eta/2 - c(2+2(K+2m)+1)]t} \geq a^{-1}t^{K}$ for all $t \geq C$.
\end{proof}

\begin{lem}\label{lem:remainer-control}
    For every $(\tilde z_0,\tilde v_0) \in \R^d \times S^{n-1}$, $\delta > 0$, $a \in (0,1/2)$, and $K \in \N_0$ there are some $C,\rho > 0$ such that for all $(z_0,v_0)$ with $$\|\tilde z_0 - z_0\| + \|\tilde v_0 - v_0\| \leq \rho \,,$$
$\epsilon < \rho$, and $t \in [C,C\epsilon^{-a}]$ we have
    $$\Prb(\|R_K^\epsilon(t)\| >  \epsilon^{(K+1)(1/2 - 2a)}e^{(\lambda + a)t}\|x_0\|) \leq \delta \,,$$
    where $\lambda \coloneqq \lambda(A^\epsilon(z_0))$ (recall the notation $R^\epsilon_K(t)$ from \eqref{eq:chaos-expansion}).
\end{lem}

\begin{proof}
    Using the notation from \eqref{eq:chaos-expansion}, we have by an identical argument to \Cref{lem:z-remains-small} that
    $$\Prb\Big(\sup_{t \leq \epsilon^{-2a}} \|\tilde z^\epsilon(t)\| > \epsilon^{1/2-2a}\Big) \lesssim \frac{\epsilon^{1-2a}}{\epsilon^{1 - 4a}} = \epsilon^{2a} \,,$$
    and for the rest of the proof we assume this event does not happen, so $\|\tilde z^\epsilon(t)\| \leq \epsilon^{1/2-2a}$ for $t \leq \epsilon^{-2a}$.
    By \Cref{lem:uniform-exponential-bound} we have $\|e^{t \tilde A}\| \lesssim e^{(\lambda + a/4)t}$ (for all $t$). Then using the mild formulation of \eqref{eq:x-def} as in \eqref{eq:chaos-expansion} we have
    $$\|x^\epsilon(t)\| \lesssim e^{(\lambda + a/4)t}\|x_0\| + \epsilon^{1/2-2a}\int_0^t e^{(\lambda + a/4)(t-s)}\|x^\epsilon(s)\|ds \,.$$
    By applying Gronwall's inequality to $f(t) \coloneqq \sup_{s \leq t} e^{-(\lambda + a/4)s}\|x^\epsilon(s)\|$ and making $\epsilon$ small enough, we obtain
    $$\|x^\epsilon(t)\| \lesssim e^{(\lambda + a/2)t}\|x_0\| \,,$$
    where the inequality is uniform over $z_0,v_0,\epsilon$ in some ball and $t \leq \epsilon^{-2a}$. Using
    \eqref{eq:chaos-expansion} and the definition of $I_J$ then gives
    $$\|R^\epsilon_K(t)\| \lesssim \epsilon^{(K+1)(1/2 - 2a)}t^{K+1}e^{(\lambda + a/2)t}\|x_0\| \,,$$
    which implies the claim (make $C$ large so that $t^{K+1}$ is a lot smaller than $e^{at/2}$ for all $t \geq C$).
\end{proof}

\subsection{Growth of \texorpdfstring{$x$}{x} implies growth of unstable modes}
The lemma below shows that the exponential decay of $x^\epsilon(t)$ is the same as the expected decay of the ``stable modes'' (the part of $v$ which lives on the nondominant eigenspaces, $v - P_\eta(z)v$) if there is very little mass on the ``unstable modes'' (the part of $v$ which lives on the dominant eigenspaces, $P_\eta(z)v$). In other words, if $x^\epsilon(t)$ grows like the unstable modes do, then the proportion of mass on these modes cannot be too small.

Recall \Cref{lem:well-posed,def:sigma-lambda,def:proj-eta,def:eta-star}. As in the previous lemmas, below $z_0$, $v_0$, and $x_0 = \|x_0\|v_0$ always refer to the initial conditions for $z^\epsilon(t)$, $v^\epsilon(t)$, and $x^\epsilon(t)$.

\begin{lem}\label{lem:high-mass-decays}
    For every $\tilde z_0 \in \R^d$, $\eta \in (0,\eta^*(\tilde z_0))$, and $C, \kappa, \rho > 0$ there is some $\delta > 0$ such that for all $v_0 \in S^{n-1}$, $z_0 \in \R^d$ with $\|\tilde z_0 - z_0\| < \delta$, and $\epsilon < \delta$ we have
    \begin{align*}
       \Prb\Big( \sup_{t \leq T_\epsilon} \|P_{\eta}(z_0)v^\epsilon(t)\| \leq \delta \text{ and } \|x^\epsilon(T_\epsilon)\| \geq e^{(\lambda-\eta+\kappa)T_\epsilon}\|x_0\|\Big) < \rho\,,
    \end{align*}
    where we set $T_\epsilon = C\ln \epsilon^{-1}$ and $\lambda \coloneqq \lambda^\epsilon(z_0) \coloneqq \lambda(A^\epsilon(z_0))$.
\end{lem}

\begin{proof}
    Denote the semigroups corresponding to the processes $(z^\epsilon(t),v^\epsilon(t))$ as $\Pp_t^\epsilon$, and for $\epsilon = 0$ we write $\Pp_t$.
    Choosing $\delta > 0$ so that $z_0 \mapsto P_\eta(z_0)$ is continuous for $\|z_0 - \tilde z_0\| < 2\delta$ (\Cref{lem:cty-of-eigenspaces}), define the compact set
    $$\mcM_\infty \coloneqq \{(z_0,v_0) \mid \|\tilde z_0 - z_0\| \leq \delta \text{ and }P_\eta(z_0)v_0 = 0\} \,.$$
    Recall \eqref{eq:L-lnr-def} and define $\lambda(z) \coloneqq \lambda(A(z))$ and $\lambda^\epsilon(z) \coloneqq \lambda(A^\epsilon(z))$. Since the $z$ component doesn't move under $\Pp_t$, we have
    \begin{equ}[e:boundLambda0]
	    \limsup_{t \to \infty} \frac{1}{t}\int_0^t (\Pp_s \Lambda)(z_0,v_0) ds \leq \lambda(z_0) - \eta\;,
    \end{equ}
    uniformly for $(z_0,v_0) \in \mcM_\infty$. (If the convergence weren't uniform over $\mcM_\infty$, we could find sequences 
    $(z_n,v_n) \in \mcM_\infty$ and $T_n \uparrow \infty$ such that \eqref{e:boundLambda0} fails along that sequence. 
    By tightness however, there is some subsequence of $\frac{1}{T_n}\int_0^{T_n} \Pp_s(z_n,v_n)ds$ which converges, and by \cite[Lem.~4.6]{extinction} the limit is an invariant measure $\mu$ on $\mcM_\infty$, but by \Cref{lem:description-of-Pinv} all such measures satisfy $\int \bigl(\Lambda(z,v) - \lambda(z) + \eta\bigr)\, d\mu(z,v) \leq 0$, thus yielding a contradiction.)

    In particular, we can find $T > 0$ such that for all $(z_0,v_0) \in \mcM_\infty$ we have
    $$\frac{1}{T} \int_0^T \Pp_s(\Lambda - \lambda)(z_0,v_0)ds \leq - \eta + \frac{\kappa}{3} \,.$$
    Since $H \coloneqq \Lambda^\epsilon(z,v) - \lambda^\epsilon(z)$ is a continuous function by \Cref{lem:cty-of-eigenspaces}, by the Feller property and \cite[Prop.~4.15]{BenaimHurth22}, using that $|z|^2/(1 + |H|)$ is proper and
    \begin{equation}\label{eq:square-bound}
        \sup_{\|z_0 - \tilde z_0\| + \epsilon \leq 3\delta } \frac{1}{T}\int_0^T (\Pp_s^\epsilon |z|^2)(z_0,v_0)ds < \infty\,,
    \end{equation}
    we conclude that
    \begin{equation}\label{eq:def-of-f}
        f(z_0,v_0,\epsilon) \coloneqq \int_0^T \Pp_s^\epsilon (\Lambda^\epsilon - \lambda^\epsilon)(z_0,v_0)ds \leq \Big(- \eta + \frac{\kappa}{2}\Big)T
    \end{equation}
    in a neighborhood of $\mcM_\infty$. In particular, by shrinking $\delta$ we may assume that \eqref{eq:def-of-f} holds if $\|z_0 - \tilde z_0\| < 2\delta$, $\epsilon < \delta$, and $\|P_\eta(z_0)v_0\| \leq \delta$.
    Over all such $z_0,v_0,\epsilon$, let $M < \infty$ be such that
    \begin{equation}\label{eq:mg-square-bound}
        \E\bigg[\Big(\int_0^T |\Lambda^\epsilon(z^\epsilon(s),v^\epsilon(s)) - \lambda^\epsilon(z^\epsilon(s))|ds\Big)^2\bigg] \leq M\,,
    \end{equation}
    which exists by $|\Lambda^\epsilon(z,v)| \lesssim 1 + |z|$, $|\lambda^\epsilon(z)| \leq \|A^\epsilon(z)\| \lesssim 1 + |z|$, and \eqref{eq:square-bound}.
    Fix $z_0,v_0,\epsilon$ with $\|z_0 - \tilde z_0\| < \delta$ and $\epsilon < \delta$.
     Write
    $$X_n \coloneqq \int_0^{nT} \bigl(\Lambda^\epsilon(z^\epsilon(s),v^\epsilon(s)) - \lambda^\epsilon(z^\epsilon(s))\bigr)\,ds = A_n + M_n$$
    where $M_n$ is a martingale and
    \begin{equation}\label{eq:tau-for-high-decay}
        \tau \coloneqq \inf\{n \geq 0 \mid \|P_\eta(z^\epsilon(nT))v^\epsilon(nT)\| > \delta \text{ or } |z^\epsilon(nT) - z_0| > \delta\} \,.
    \end{equation}
    Explicitly, with $\Delta_k \coloneqq \int_{kT}^{(k+1)T} \bigl(\Lambda^\epsilon(z^\epsilon(s),v^\epsilon(s)) - \lambda^\epsilon(z^\epsilon(s))\bigr)\,ds$, we have
    \begin{equ}
        A_n = \sum_{k=0}^{n-1} f(z^\epsilon(kT), v^\epsilon(kT), \epsilon) \;,\quad
        M_n = \sum_{k=0}^{n-1} \bigl(\Delta_k - \E[\Delta_k \mid \F_{kT}]\bigr) \,.
    \end{equ}
    Then we have by \eqref{eq:def-of-f} that
    \begin{equation}\label{eq:an-inequality}
        A_{n \wedge \tau} \leq (n \wedge \tau)T\Big(-\eta + \frac{\kappa}{2}\Big) \,.
    \end{equation}
    Since $\E[M_{n \wedge \tau}^2] \leq \sum_{k=0}^{n-1} \E[\Id_{\tau > k}\Delta_k^2]$, the Markov property, \eqref{eq:mg-square-bound}, and Doob's inequality yield
    \begin{equation}\label{eq:def-of-bad-event-a}
        \Prb(\CA) \coloneqq \Prb\Big(\sup_{n \leq T_\epsilon/T}|M_{n \wedge \tau}| > \frac{\kappa T_\epsilon}{8}\Big) \leq  M\frac{T_\epsilon}{T}\Big(\frac{\kappa}{8}T_\epsilon\Big)^{-2} \,.
    \end{equation}
    By \eqref{eq:dlnr}
    \begin{equation}\label{eq:ln-xepsilon-inequality}
        \begin{aligned}
            \ln \|x^\epsilon((n\wedge \tau)T)\| - \ln\|x_0\| &= X_{n\wedge \tau} + \int_0^{(n\wedge \tau)T} \lambda^\epsilon(z^\epsilon(s))ds \\
            &\leq X_{n\wedge \tau} + (n\wedge \tau)\lambda T + \frac{\kappa}{4}nT
        \end{aligned}
    \end{equation}
    when $\delta$ is small enough so that $|z - z_0| \leq \delta$ implies $|\lambda^\epsilon(z) - \lambda| = |\lambda^\epsilon(z) - \lambda^\epsilon(z_0)| \leq \kappa/4$.
    If $\CA^c$ and $\CB \coloneqq \{\tau T > T_\epsilon\}$ occur, by \eqref{eq:dlnr}, \eqref{eq:ln-xepsilon-inequality}, and \eqref{eq:an-inequality} we have, with $n$ such that $nT \leq T_\epsilon <(n+1)T$,
    $$ \ln \|x^\epsilon(T_\epsilon)\| \leq \ln\|x_0\| + \Big(\lambda - \eta + \frac{7}{8}\kappa\Big)T_\epsilon + \int_{nT}^{T_\epsilon} |\Lambda^\epsilon(z^\epsilon(s), v^\epsilon(s))|ds + (|\lambda| + \eta)T \,.$$
    Since $\epsilon$ can be taken to be small and $\Lambda$ is continuous, we obtain
    \begin{equation}\label{eq:final-contradiction}
        \|x^\epsilon(T_\epsilon)\| \leq e^{(\lambda-\eta+15\kappa/16)T_\epsilon}\|x_0\| \,.
    \end{equation}
Since \eqref{eq:def-of-bad-event-a} implies $\Prb(\CA^c) \to 1$ as $\epsilon \downarrow 0$ and \Cref{lem:z-remains-small} implies $$\lim_{\epsilon \downarrow 0}\Prb\Big(\sup_{t \leq T_\epsilon} \|P_{\eta}(z_0)v^\epsilon(t)\| \leq \delta \text{ and } \CB^c\Big) = 0 \,,$$
  up to an arbitrarily small probability event we have $\sup_{t \leq T_\epsilon} \|P_{\eta}(z_0)v^\epsilon(t)\| \leq \delta$ implies \eqref{eq:final-contradiction}, which prevents $\|x^\epsilon(T_\epsilon)\| \geq e^{(\lambda-\eta+\kappa)T_\epsilon}\|x_0\|$, proving the claim.
\end{proof}

\subsection{Proof of \texorpdfstring{\Cref{lem:mass-from-high-to-low}}{Lemma}}
\label{sec:firstLemma}

The intuition is that \Cref{lem:mass-generation,lem:remainer-control} imply that $x$ grows
while \Cref{lem:high-mass-decays} shows that this then forces the 
unstable modes to grow, which is essentially what \Cref{lem:mass-from-high-to-low} quantifies. 

To make this precise, let $(\tilde z_0, \tilde v_0) \in \R^d \times S^{n-1}$ and $\eta \in (0,\eta^*(\tilde z_0))$ be fixed, 
let $K \in \N_0$ be as in \Cref{lem:mass-generation}, and let $C \coloneqq (K+1)/(2\eta)$. We also set $T_\epsilon = C \ln \epsilon^{-1}$ as above. Then by \Cref{lem:mass-generation} there exists $\delta, \rho > 0$ such that for all $(z_0,v_0,\epsilon)$ within distance $\rho$ of $(\tilde z_0, \tilde v_0, 0)$ we have
    \begin{equation}\label{eq:lower-bound-on-xk}
        \Prb(\|X_K^\epsilon(T_\epsilon)\| >  \epsilon^{K/2}e^{(\lambda - \eta/(2K+2))T_\epsilon}\|x_0\|) \geq \delta
    \end{equation}
    (where the notation is explained in \Cref{lem:mass-generation}). Choose $a \in (0,1/2)$ small enough that
    $$1/4 < 1/2-2a-2aK-\frac{a(K+1)}{2\eta} \,.$$
    Then by \Cref{lem:remainer-control} (after possibly reducing $\rho$) we have
    \begin{equation}\label{eq:upper-bound-on-rk}
        \Prb(\|R_K^\epsilon(T_\epsilon)\| >  \epsilon^{(K+1)(1/2 - 2a)}e^{(\lambda + a)T_\epsilon}\|x_0\|) \leq \frac{1}{2}\delta \,.
    \end{equation}
    Note that $e^{(\lambda + a)T_\epsilon} = \epsilon^{-C(\lambda+a)}$ and that the power of $\epsilon^{-1}$ in \eqref{eq:upper-bound-on-rk} is smaller than the one in \eqref{eq:lower-bound-on-xk}:
    \begin{align*}
        -(K+1)(1/2 - 2a)+(\lambda + a)C &= \lambda C  -\frac{K}{2} - \Big(1/2-2a-2aK-\frac{a(K+1)}{2\eta}\Big) \\
        &< \lambda C - \frac{K}{2} - \frac{1}{4} \\
        &= -\frac{K}{2}+(\lambda - \eta/(2K+2))C \,,
    \end{align*}
   so we conclude (for $\epsilon$ small enough) that
    $$\Prb\Big(\|x^\epsilon(T_\epsilon)\| \geq  \epsilon^{K/2}e^{(\lambda - 2\eta/(3K+3))T_\epsilon}\|x_0\|\Big) \geq \frac{1}{2}\delta \,.$$
    Since the factor in front of $T_\epsilon$ above is
    $$-\frac{K}{2C} + \lambda - 2\eta/(3K+3) = \lambda-\eta\Big(\frac{K}{K+1} + \frac{2}{3}\frac{1}{K+1}\Big) = \lambda -\eta +\kappa$$
    for some $\kappa > 0$, we conclude by \Cref{lem:high-mass-decays} (set $\rho = \delta/4$).\qed

\subsection{Proof of \texorpdfstring{\Cref{lem:mass-stays-on-high}}{Lemma}}
\label{sec:secondLemma}

We need to show that if we have a bit of mass on the unstable modes then (projectively) that mass will endure for a long time. Intuitively, this is because the unstable modes grow at a faster exponential rate than the stable ones.

By \Cref{lem:z-remains-small} we see that, given any $\delta >0$, we can find $\rho$ such that,
with probability at least $1-\delta$, one has 
$\|z^\epsilon(t) - \tilde z_0\| \leq 2\rho^{1/8}$
 for all $t \leq \epsilon^{-1/2}$ provided
that $\epsilon < \rho$ and $\|z^\epsilon(0) - \tilde z_0\|\le \rho$. 
By \Cref{lem:cty-of-eigenspaces}, the map $z \mapsto P_\eta(z)$ is continuous on $\{z\,:\, \|z - z_0\| \leq 2\rho^{1/8}\}$ 
(by possibly decreasing the value of $\rho$). It therefore remains to show, by possibly choosing
$\rho$ even smaller, the deterministic implication that if $z$ is any continuous function such that 
$\sup_{t \le \epsilon^{-1/2}}\|z(t) - \tilde z_0\| \le 2\rho^{1/8}$,
 and   $\|P_\eta(z_0)v_0\| \geq \delta$, then one has
    $$\inf_{T \leq t \leq \epsilon^{-1/2}} \frac{\|P_{\eta}(z_0)x_t\|}{\|x_t\|} \ge \frac{\delta}{3}\;,$$
for every solution $x^\epsilon$ to the ODE $\dot x^\epsilon = A^\epsilon(z_t)x_t^\epsilon$.

Note first that the triangle inequality $\|x\| \le \|x-y\|+\|y\|$ implies that
\begin{equ}
\frac{\|y\|}{\|x-y\|} \ge \delta \implies \frac{\|y\|}{\|x\|} \ge \frac{\delta}{1+\delta}\;,
\end{equ}
for any two vectors $x,y$ in a Banach space. In particular, this shows that if $\delta \le 1$ then one has
\begin{equation}\label{eq:going-to-c}
    \frac{\|P_\eta(z_0)x\|}{\|x - P_\eta(z_0)x\|} \geq \delta \implies \frac{\|P_\eta(z_0)x\|}{\|x\|} \geq  \frac{\delta}2\;.
\end{equation}
We may choose $T > 0$ large enough (of the order $\log \delta^{-1}$)
such that for all $v \in S^{n-1}$ with $\|P_\eta(\tilde z_0)v\| \geq \delta/2$ and $t \in [T, 2T]$ we have
\begin{equ}[e:boundProjAv]
    \|P_\eta(\tilde z_0)e^{tA(\tilde z_0)}v\| \geq 2e^{(\lambda(\tilde z_0) -\eta)t} \quad \text{and} \quad \|e^{tA(\tilde z_0)}v - P_\eta(\tilde z_0)e^{tA(\tilde z_0)}v\| \leq \frac{1}{2}e^{(\lambda(\tilde z_0) -\eta)t} \,.
\end{equ}
(The above is elementary to prove using Jordan canonical form and Gelfand's lemma, see for example the proofs of \Cref{lem:uniform-exponential-bound,lem:uniform-exp-lower-bound}.) By continuity of $P_\eta$ and compactness of $S^{n-1}$,
it follows that $\|P_\eta(z_0)v\| \geq \delta$ implies $\|P_\eta(\tilde z_0)v\| \geq \delta/2$ provided that $\rho$ is
chosen sufficiently small and we can replace $A(\tilde z_0)$ by $A^\epsilon(z_0)$ and $P_\eta(\tilde z_0)$ by $P_\eta(z_0)$ in \eqref{e:boundProjAv}
(modulo losing a constant factor in both inequalities), and therefore
    \begin{equation}\label{eq:ratios-below-and-above}
         \|P_\eta(z_0)v\| \geq \delta \implies \frac{\|P_\eta(z_0)e^{tA^\epsilon(z_0)}v\|}{\|e^{tA^\epsilon(z_0)}v - P_\eta(z_0)e^{tA^\epsilon(z_0)}v\|} \geq 2 \,.
    \end{equation}
    
    Recall that the variation of constants applied to \eqref{eq:x-def} yields
    $$x^\epsilon(t) = e^{tA^\epsilon(z_0)}x_0 + \sum_{i=1}^d \int_0^t e^{(t-s)A^\epsilon(z_0)}B_ix^\epsilon(s)(z_i(s)-z_0)\,ds \,.$$
    Since $\|z(t)-z_0\| \leq \rho^{1/8}$, it follows from Gronwall's inequality that there is some constant $C > 0$ (depending 
    on $T$ as above and therefore on $\delta$, but not on $\rho$) such that
    $$\Big\|\sum_{i=1}^d \int_0^t e^{(t-s)A^\epsilon(z_0)}B_ix^\epsilon(s) z_i(s)\,ds\Big\| \leq C\rho^{1/8}\|x_0\| \,.$$
    By making $\rho$ small enough, we can combine the above with \eqref{eq:ratios-below-and-above} 
    and the fact that, by \eqref{e:boundProjAv}, the numerator in this expression is lower bounded by some constant
    $m$ (depending on $T$)
    to conclude that
    $$\frac{\|P_\eta(z_0)x^\epsilon(0)\|}{\|x^\epsilon(0)\|} \geq \delta \implies \frac{\|P_\eta(z_0)x^\epsilon(t)\|}{\|x^\epsilon(t) - P_\eta(z_0)x^\epsilon(t)\|} \geq 1 \,,$$
    and by continuity (recall the numerator is bounded below) and \eqref{eq:going-to-c} we obtain
    $$\|P_\eta(z_0)v_0\| \geq \delta \implies \|P_\eta(z(t))v^\epsilon(t)\| \geq \frac13 \ge 2\delta \,,$$
    uniformly over $t \in [T,2T]$, where we set $v^\epsilon = x^\epsilon / \|x^\epsilon\|$. This bound can now be iterated 
    to show that $\|P_\eta(z(t))v^\epsilon(t)\| \ge 2\delta$ for all $t \in [T, \epsilon^{-1/2}]$,
    and the claim follows since, by continuity, $\|(P_\eta(z(t))- P_\eta(z_0))v^\epsilon(t)\| \le \delta$
    provided that $\rho$ is small enough.

\section{Appendix}\label{sec:appendix}
\begin{lem}\label{lem:z-remains-small}
    For every $M \geq 0$ there is some $C_M > 0$ such that for all $z_0 \in \R^d$ with $|z_0| \leq M$,
    $$\Prb\Big(\sup_{t \leq \epsilon^{-2/3}} \|z^\epsilon(t) - z_0\| > \epsilon^{1/8}\Big) \leq C_M\epsilon^{\frac{1}{12}} \,.$$
\end{lem}
\begin{proof}
    Fix $M > 0$. Recall \eqref{eq:zvr-system-def} and define $\tau \coloneqq \inf\{t \geq 0 \mid \|z^\epsilon(t)\| > M+1\}$. For any $T > 0$, by Doob's martingale inequality we have
    $$\E\Big[\sup_{t \leq T \wedge \tau} \|z^\epsilon(t) - z_0\|^2\Big] \lesssim \epsilon^2T^2 + \epsilon T \,.$$
    Setting $T = \epsilon^{-2/3}$, we have
    $$\Prb\Big(\sup_{t \leq \epsilon^{-2/3} \wedge \tau} \|z^\epsilon(t) - z_0\| > \epsilon^{1/8}\Big) \lesssim \frac{\epsilon^2\epsilon^{-4/3} + \epsilon \epsilon^{-2/3}}{\epsilon^{1/4}} \lesssim \epsilon^{1-2/3-1/4} = \epsilon^{\frac{1}{12}}$$
    (recall $\epsilon \in [0,1]$). By continuity of the sample paths,
    $$\tau < \epsilon^{-2/3} \implies \sup_{t \leq \epsilon^{-2/3} \wedge \tau} \|z^\epsilon(t) - z_0\| \geq 1 > \epsilon^{1/8} \,,$$
    so
    $$\Prb\Big(\sup_{t \leq \epsilon^{-2/3} \wedge \tau} \|z^\epsilon(t) - z_0\| > \epsilon^{1/8}\Big) = \Prb\Big(\sup_{t \leq \epsilon^{-2/3}} \|z^\epsilon(t) - z_0\| > \epsilon^{1/8}\Big) \,,$$
    which finishes the claim.
\end{proof}

\begin{lem}\label{lem:cty-of-eigenspaces}
    Recall \Cref{def:sigma-lambda,def:eta-star}.
    \begin{itemize}
         \item $T \mapsto \lambda(T)$ is continuous on all of $\R^{n \times n}$.
         \item If $z_0 \in \R^d$ and $\eta \in (0,\eta^*(z_0))$ then the function $z \mapsto P_\eta(z)$ as defined in \Cref{def:proj-eta} is continuous in a neighborhood of $z_0$.
    \end{itemize}
\end{lem}
\begin{proof}
This is standard, see for example \cite[Thm~II.5.1]{Kato}.
\end{proof}

\begin{lem}\label{lem:taylor-expansion}
    Let $K \in \N$, $m \in \N_0$, $M \geq \delta > 0$, and $f:\Delta^K(1) \to \R^n$ be smooth and such that
    $$\|D^m f(0)\| \geq \delta \quad \text{and} \quad \sup_{t \in \Delta^K(1)} \|D^{m+1} f(t)\| \leq M\,.$$
    Then there is some $C(K,m)$ (independent of $f$) such that
    $$\|f\|_{L^2(\Delta^K(1))} \geq C(K,m)\delta^{K/2+m+1}M^{-K/2-m} \,.$$
\end{lem}

\begin{proof}
We may assume $\delta = 1$ since the general case follows by replacing $f$ with $\delta^{-1}f$.
    Write the Taylor expansion of $f$:
    $$f(t_1,\dots,t_K) = \sum_{j=0}^m \frac{1}{j!}\sum_{I \in [K]^j} \partial_If(0)\prod_{i=1}^j t_{I_i} + R(t_1,\dots,t_K)$$
    where $|R(t_1,\dots,t_K)| \leq \frac{M}{(m+1)!}|t|^{m+1}$ and here $|t| = \max_i t_i$ and $\partial_I$ denotes $\prod_{l=1}^j \partial_{I_l}$. Then for all $\rho \in [0,1]$
    $$\|R\|_{L^2(\Delta^K(\rho))} \leq \frac{M}{\sqrt{K!}(m+1)!}\rho^{K/2+m+1} \,.$$
    Using the substitution $s_i = \rho^{-1}t_i$, we have
    $$\Big\|\sum_{j=0}^m \frac{1}{j!}\sum_{I \in [K]^j} \partial_If(0)\prod_{i=1}^j t_{I_i}\Big\|^2_{L^2(\Delta^K(\rho))} = \rho^K \Big\|\sum_{j=0}^m \frac{\rho^j}{j!}\sum_{I \in [K]^j} \partial_If(0)\prod_{i=1}^j s_{I_i}\Big\|^2_{L^2(\Delta^K(1))} \,.$$
    Since the polynomials of degree at most $m$ are linearly independent in $L^2(\Delta^K(1))$, there is some $c_{K,m} > 0$ such that
    $$\Big\|\sum_{j=0}^m \frac{1}{j!}\sum_{I \in [K]^j} \partial_If(0)\prod_{i=1}^j t_{I_i}\Big\|_{L^2(\Delta^K(\rho))} \geq c_{K,m}\rho^{K/2+m} \,,$$
    so lower bounding $\|f\|_{L^2(\Delta^K(1))}$ by $\|f\|_{L^2(\Delta^K(\rho))}$ for an optimal choice of $\rho \in [0,1]$ we conclude
    \begin{equation*}
        \|f\|_{L^2(\Delta^K(1))} \geq \rho^{K/2+m}\Big(c_{K,m} - \frac{M}{\sqrt{K!}(m+1)!}\rho\Big) \geq C(K,m,M)
    \end{equation*}
    A short computation shows that the supremum occurs at
    $$\rho \coloneqq 1 \wedge \frac{K/2+m}{K/2+m+1}\frac{\sqrt{K!}(m+1)!}{M}c_{K,m} \,,$$
    which gives a value of
    $C(K,m,M) = C'(K,m)M^{-K/2-m}$ when $M$ is larger than some $M^* > 0$. If $M \in [1,M^*]$ we have $C(K,m,M) \geq C(K,m) \geq C(K,m)M^{-K/2-m}$ for some $C(K,m) > 0$, which finishes the claim (recall $\delta = 1$).
\end{proof}

\begin{lem}\label{lem:algebra-prop}
    Under \Cref{as:nondegen}, for every $(z,v) \in \R^d \times S^{n-1}$ and $\eta > 0$ there are $K \in \N_0$, $J \in [d]^K$, and $k_1,\dots,k_K \in \N_0$ such that
    $$\Big\|P_\eta(z)\prod_{j=1}^K \operatorname{ad}_{A(z)}^{k_j}(B_{J_j})v\Big\| > 0$$
    (see \Cref{def:proj-eta}).
\end{lem}

\begin{proof}
    Define a subspace $W$ of $\R^n$ as the intersection of the set of all $x \in \R^n$ such that $P_\eta(z)\prod_{j=1}^K \operatorname{ad}_{A(z)}^{k_j}(B_{J_j})x = 0$, where the intersection runs over all choices of $K \in \N_0$, $J \in [d]^K$, and $k \in \N_0^K$. By definition, this subspace is invariant for
    $$S(z) \coloneqq \{\operatorname{ad}_{A(z)}^k(B_i) \mid k \geq 0, i = 1,\dots,d\} \,.$$
    Since $A(z) = A + \sum_{i=1}^d z_iB_i$ we have $\operatorname{ad}_{A} = \operatorname{ad}_{A(z)} - \sum_{i=1}^d z_i\operatorname{ad}_{B_i}$, so
    $$\operatorname{ad}_A^k(B_i) \in \operatorname{Span}\Big(\Big\{\prod_{j=1}^m \operatorname{ad}_{A(z)}^{k_j}(B_{J_j}) \Bigm\vert m \leq k+1, J \in [d]^m, k \in \N_0^m\Big\}\Big) \subset \operatorname{Alg}(S(z))$$
    (where $\operatorname{Alg}(S(z))$ is the algebra generated by $S(z)$) because $\operatorname{ad}_M(\prod_{j=1}^m M_j) = \sum_{j=1}^m M_1\dots M_{j-1}\operatorname{ad}_M(M_j)M_{j+1} \dots M_m$. It follows that $W$ is invariant for all $\operatorname{ad}_A^k(B_i)$, and thus for $S$ (as defined in \Cref{as:nondegen}). By \Cref{as:nondegen} this implies that $W = \{0\}$ or $W = \R^n$. Since $W$ is a subset of the kernel of $P_\eta(z)$ (take $K = 0$) and $P_\eta(z)$ is by definition nonzero, we conclude that $W \neq \R^n$ and thus $W = \{0\}$, which proves the claim.
\end{proof}

The next two lemmas allow us to 
bound expressions of the form  
     \begin{equs}
     I_k(f) &= \int_{[0,t]^k} \scal{f(t_1,\ldots,t_k),z^\eps(t_1)\otimes\dots\otimes z^\eps(t_k)}\,dt_1\ldots dt_k \label{e:defIk}\\
     &= k!\int_{\Delta^k(t)} \scal{(\Pi f)(t_1,\ldots,t_k),z^\eps(t_1)\otimes\dots\otimes z^\eps(t_k)}\,dt_1\ldots dt_k\;,
     \end{equs}
where we view $I_k$ as an operator $I_k \colon \CH^{\otimes k} \to L^2(\Omega,\P)$ with $\CH = L^2([0,t],\R^d)$,
in terms of the homogeneous Wiener chaos
     \begin{equ}[e:defIktilde]
     \tilde I_k(f) = k!\int_{\Delta^k(t)} \scal{(\Pi f)(t_1,\ldots,t_k),dW(t_1)\otimes\dots\otimes dW(t_k)}\,dt_1\ldots dt_k\;.
     \end{equ}
     Here $\Pi \colon \CH^{\otimes k} \to \CH^{\otimes_s k}$ denotes the symmetrisation operator.
     We first show the following lower bound.

\begin{lem}\label{lem:paley-zygmund}
Given $K \in \N_0$ and a sequence of elements $g_k \in \CH^{\otimes_s k}$, let
    $X = \sum_{k=0}^K \tilde I_k(g_k)$.
Then, whenever $\|g_K\| \geq \delta$, one has
    $$\Prb\Big(|X| > \Big[\frac{K!}{2}\Big]^{1/2}\delta\Big) \geq \frac{3^{-2K}}{4} \,.$$
\end{lem}

\begin{proof}
    Set $t = \frac{1}{2}\ln 3$, so that $q(t) = e^{2t}+1 = 4$, and write $X_k = \tilde I_k(g_k)$ for the homogeneous components of $X$. 
    Then, by hypercontractivity of the Ornstein--Uhlenbeck semigroup $T_t$ (see for example \cite[Def.~5.1.1, Thm~5.1.3]{malliavin}), we have
    $$\|X\|_{L^4}^2 = \|T_t\tilde X\|_{L^4}^2 \leq \|\tilde X\|_{L^2}^2 \leq e^{2Kt}\|X\|_{L^2}^2 = 3^K\|X\|_{L^2}^2  \,,$$
    where $\tilde X = \sum_{k=0}^K e^{kt}X_k$ and the last inequality is by the orthogonality of the Wiener chaoses, see
    \cite[Eq.~4.1]{malliavin}. For $\theta \in [0,1]$, the Paley--Zygmund inequality implies
    $$\Prb(|X| > \theta\|X\|_{L^2}) = \Prb(X^2 > \theta^2\|X\|_{L^2}^2) \geq (1-\theta^2)^2\frac{\|X\|_{L^2}^4}{\|X\|_{L^4}^4} \geq (1-\theta^2)^23^{-2K} \,.$$
    By the isometric property of the $\tilde I_k$ \cite[Eq.~4.1]{malliavin} we have
    $$\|X\|_{L^2}^2 = \sum_{k=0}^K  k! \|g_k\|^2 \ge \delta^2 K! \,,$$
    so the claim follows by setting $\theta^2 = 1/2$.
\end{proof}

\begin{lem}\label{lem:chaos-terms-l2-lower-bound}
    Let $z_0 \in \R^d$, $\epsilon \in (0,1]$,
    $t \geq 1$ be arbitrary and recall \eqref{e:defIk}. Let
    $$X_t \coloneqq \sum_{k=0}^K  I_k(f_k) \,,$$
    where $K \in \N_0$ and $f_k \in \CH^{\otimes_s k}$. Then there exist $C_K > 0$ (depending only on $K$)
    such that
    $$\Prb\Big(|X_t| \geq C_K\epsilon^{K/2}t^{-K}\|f_K\|^2\|f_K\|_{H^{2K}}^{-1}\Big) \geq \frac{3^{-2K}}{4} \,,$$
    where we write $H^m$ for the Sobolev space of order $m$.
\end{lem}
\begin{proof}
The claim is trivial for $K=0$ so we assume $K \ge 1$.
    Recall the notation from \Cref{lem:paley-zygmund} and write the solution to \eqref{eq:zvr-system-def} as
    \begin{equ}[e:defzeps]
     z^\epsilon(t) = e^{-\epsilon D t}z_0 + \epsilon^{1/2}\int_0^t e^{-\epsilon D(t-s)}dW_s \;.
    \end{equ}
     
      Write furthermore $G \colon \CH \to \CH$ for the operator
     \begin{equ}[e:defG]
     (Gf)(r) = \int_r^t e^{-\eps D (s-r)}f(s)\,ds \;.
     \end{equ}
     It the follows from \cite[Equ.~4.2]{malliavin} and \eqref{e:defzeps} that
     \begin{equ}[e:exprIk]
     I_k(f) = \eps^{k/2} \tilde I_k(G^{\otimes k} f) + R\;,
     \end{equ}
	where $R$ belongs to the chaoses of order $k-1$ and lower. It is straightforward to verify that 
	one has $Gf \in H^1$ for every $f \in L^2$ and that the operator $Lf = \epsilon D f - f'$ with
	domain $\CD(L) = H^1$ is a left inverse for $G$.

	It follows immediately that
	\begin{equ}[e:assumptionG]
	\|Lf\|^2 \lesssim \|f\| \|\Lambda f\|\;,
	\end{equ}
	for all $f \in \CD(\Lambda)$, 
	where $\Lambda$ denotes the positive selfadjoint operator such that $\|\Lambda f\|^2 = \|f\|_{H^2}^2$.
	
	By Cauchy--Schwarz, \eqref{e:assumptionG} implies that, for any Hilbert space $\bar\CH$ and $f \in \CD(\Lambda) \otimes \bar \CH$,
	one has
	\begin{equ}
	\|(L \otimes \id)f\|_{\CH \otimes \bar \CH}^2 \lesssim \|f\|_{\CH \otimes \bar \CH} \|(\Lambda \otimes \id) f\|_{\CH \otimes \bar \CH}\;.
	\end{equ}
	Iterating this inequality and using the fact that $\|(\Lambda^{\otimes \ell} \otimes \id^{\otimes(k-\ell)})f\|_{L^2} \lesssim \|f\|_{H^{2\ell}}$, we obtain the bound
	\begin{equ}
	\|L^{\otimes k}f\|^{2^k} \lesssim \prod_{0 \le \ell \le k} \|f\|_{H^{2\ell}}^{\binom{k}{\ell}}
	\lesssim \|f\|^{2^{k-1}}\|f\|_{H^{2k}}^{2^{k-1}}\;,
	\end{equ}
	where the second inequality follows from the interpolation inequality $\|f\|_{H^{2\ell}}^k \lesssim \|f\|_{H^{2k}}^\ell \|f\|^{k-\ell}$. We conclude that
	\begin{equ}
	\|G^{\otimes k} f\| \gtrsim \frac{\|f\|^2}{\|G^{\otimes k} f\|_{H^{2k}}}
	\gtrsim \frac{\|f\|^2}{\|f\|_{H^{2k}}}\;,
	\end{equ}
	where we used the uniform boundedness of $G$ on $H^{2k}$.
	The required bound then follows immediately from \eqref{e:exprIk} and \Cref{lem:paley-zygmund}.
\end{proof}

\begin{lem}\label{lem:uniform-exponential-bound}
    For every $c > 0$, $\tilde z_0 \in \R^d$ there is some $\delta,C > 0$ such that for all $z_0 \in \R^d$, $\epsilon > 0$ with $\|z_0 - \tilde z_0\| \leq \delta$, $\epsilon < \delta$, and $t \geq 0$ we have
    $$\|e^{t\tilde A}\| \leq Ce^{(\lambda+c)t} \,,$$
    where $\tilde A \coloneqq A^\epsilon(z_0)$ and $\lambda \coloneqq \lambda(A^\epsilon(z_0))$.
\end{lem}
\begin{proof}
By \Cref{lem:cty-of-eigenspaces} we may instead prove the theorem with $\lambda = \lambda(A(\tilde z_0))$, which is the notation we stick to in the proof. Let $A \coloneqq A(\tilde z_0)$. By Gelfand's lemma
    $$\lim_{t \to \infty} \|e^{tA}\|^{1/t} = e^{\lambda} \,,$$
    so there is some $\tilde C > 0$ such that
    $$\|e^{tA}\| \leq \tilde Ce^{(\lambda + c/2)t} \,.$$
    Since $e^{t\tilde A}$ is the solution to $M' = \tilde AM = AM + (\tilde A - A)M$,
    $$e^{t \tilde A} = e^{tA} + \int_0^t e^{(t-s)A}(\tilde A - A)e^{s \tilde A}ds \,.$$
    By continuity of $(t,z_0,\epsilon) \mapsto e^{t\tilde A}$ there is some $C > \tilde C$ such that $\|e^{s \tilde A}\| < C e^{(\lambda + c)s}$ for all $s \leq 1$ and $\|z_0 - \tilde z_0\| \leq 1$, $\epsilon \leq 1$. In particular, the time $\tau \coloneqq \inf\{t > 0 \mid \|e^{t \tilde A}\| > C e^{(\lambda + c)t} \}$ is strictly positive. If $\tau < \infty$ then
    \begin{align*}
        Ce^{(\lambda+c)\tau} =\|e^{\tau \tilde A}\| &\leq \tilde Ce^{(\lambda +c/2)\tau} + \|\tilde A - A\|\tilde CC\int_0^\tau e^{(\tau - s)(\lambda + c/2) + s(\lambda + c)}ds \\
        &= Ce^{(\lambda + c)\tau}\Big[\frac{\tilde C}{C}e^{-c\tau/2} + \frac{2\|\tilde A - A\|\tilde C}{c}(1-e^{-c \tau/2})\Big] \,.
    \end{align*}
    By shrinking $\delta$, we have that $\|\tilde A - A\| < c/(2C)$, and thus
    $$Ce^{(\lambda+c)\tau} \leq Ce^{(\lambda + c)\tau}\frac{\tilde C}{C} < Ce^{(\lambda +c)\tau}$$
    by $C > \tilde C$, contradicting $\tau < \infty$. By our definition of $\tau$ the claim is proven.
\end{proof}

\begin{lem}\label{lem:uniform-exp-lower-bound}
    For every $\tilde z_0 \in \R^d$, $\eta \in (0,\eta^*(\tilde z_0))$, and $\delta, c > 0$ there are $\rho, C > 0$ such that for all $z_0 \in \R^d$ with $\|z_0-\tilde z_0\| \leq \rho$, $\epsilon < \rho$, $v \in \R^n$, and $t \geq C$
    $$\|P_\eta(z_0)v\| \geq \delta \|v\| \implies \|P_{\eta}(z_0)e^{t A^\epsilon(z_0)}v\| \geq e^{(\lambda(A^\epsilon(z_0))-c)t}\|v\| \,,$$
    where $P_\eta(z_0)$ is as in \Cref{def:proj-eta}.
\end{lem}
\begin{proof}
For fixed $z_0 = \tilde z_0$ and $\epsilon = 0$ the claim would be immediate by decomposing $v = P_\eta(\tilde z_0)v + [v - P_\eta(\tilde z_0)v]$ and noting that $e^{tA(\tilde z_0)}$ acts as approximately $e^{\lambda(A(\tilde z_0))t}$ on $P_\eta(\tilde z_0)v$ and at most $e^{[\lambda(A(\tilde z_0)) - \eta]t}$ on $v - P_\eta(\tilde z_0)v$. To make this argument uniform in $z_0,\epsilon$, we first fix a large enough $C$ and use the argument above for $t \in [C,2C], z_0 = \tilde z_0, \epsilon = 0$, then choose $\rho$ small enough so that by continuity similar bounds hold uniformly for $t \in [C,2C], \|z_0 -\tilde z_0\| \leq \rho, \epsilon < \rho$, and finally iterate these bounds to handle arbitrary $t \geq C$. The details are given below:

    We may assume that $c < \eta^*(\tilde z_0)$ because if the result holds for $c$ then it is true for all larger $c$. It follows from \Cref{lem:cty-of-eigenspaces} that by choosing an appropriate $\rho$ we can assume $\|P_\eta(z_0) - P_\eta(\tilde z_0)\|$ is as small as we want, in particular smaller than $\delta/2$.

    Then by our assumption $\|P_\eta(z_0)v\| \geq \delta\|v\|$ we conclude $\|P_\eta(\tilde z_0)v\| \geq \delta\|v\|/2$. We choose an equivalent norm so that the generalized eigenspaces of $A(\tilde z_0)$ are orthogonal so that, using the notation in \Cref{def:proj-eta},
    $$\sum_{\{i \mid \Re \lambda_i = \lambda(A(\tilde z_0))\}} \|e^{t(\lambda_i + N_i)}v_i\|^2 = \|P_\eta(\tilde z_0)e^{tA(\tilde z_0)}v\|^2$$
    (since $\eta < \eta^*(\tilde z_0)$), where $N_i$ is the nilpotent part of the Jordan block. Set $\lambda \coloneqq \lambda(A(\tilde z_0))$. Since there are only finitely many options for $N_i$, by Gelfand's formula ($\lim_{t \to \infty} \|e^{-tN_i}\|^{1/t} = \lambda(e^{-N_i}) = 1$) and $\|e^{tN_i}v_i\| \geq \|e^{-tN_i}\|^{-1}\|v_i\|$ we have that
    $$\frac{\delta}{2}e^{t(\lambda - c/4)}\|v\| \leq e^{t(\lambda - c/4)}\|P_\eta(\tilde z_0)v\| \lesssim \|P_\eta(\tilde z_0)e^{tA(\tilde z_0)}v\| \,.$$
    Thus, for $C$ large enough and all $t \geq C$ we have
    $$\|P_\eta(\tilde z_0)e^{t A(\tilde z_0)}v\| \geq 3e^{(\lambda-c/2)t}\|v\| \,.$$
    We have again by Gelfand's lemma and $c < \eta^*(\tilde z_0)$ that there is $K > 0$ such that for all $t \geq 0$ (in particular, $K$ is independent of $C$)
    $$\|(I - P_\eta(\tilde z_0))e^{tA(\tilde z_0)}v\| \leq Ke^{(\lambda - c)t}\|v\| \,.$$
    By the continuity mentioned above we obtain for $t \in [C,2C]$ that
    \begin{equation*}
        \|P_\eta(z_0)e^{t A^\epsilon(z_0)}v\| \geq 2e^{(\lambda-c/2)t}\|v\| \quad \text{and} \quad  \|(I - P_\eta(z_0))e^{tA^\epsilon(z_0)}v\| \leq 2Ke^{(\lambda - c)t}\|v\| \,.
    \end{equation*} 
    (where $I$ denotes the identity operator). If we choose $C$ large enough, by triangle inequality and our equivalent norm (so $\|P_\eta(\tilde z_0)\| \leq 1$) we have (for all $t \in [C,2C]$)
    \begin{equation*}
    \|P_\eta(z_0)e^{tA^\epsilon(z_0)}v\| \geq \delta\|e^{tA^\epsilon(z_0)}v\| \quad \text{and} \quad \|e^{tA^\epsilon(z_0)}v\| \geq e^{(\lambda - c/2)t}\|v\| \,.
    \end{equation*}
    By \Cref{lem:cty-of-eigenspaces} we may assume $|\lambda - \lambda(A^\epsilon(z_0))| \leq c/2$, so the claim is proven by iterating the bounds above for $t \in [2C,3C]$ (replace $v$ with $e^{CA^\epsilon(z_0)}v$), $t \in [3C,4C]$ (replace $v$ with $e^{2CA^\epsilon(z_0)}v$), etc.
\end{proof}

\begin{thm}\label{thm:persistence-extinction}
    Let $(z_t,x_t) \in \mcM \coloneqq \R^d \times \R^n$ denote the solution to
    \begin{equ}\label{eq:general-sde}
            dz = f(z,x)dt + \sigma dW_t \,,\qquad
            dx = F(z,x)dt \,,
    \end{equ}
    where $W_t$ is a standard $\R^d$-valued Wiener process, $\sigma \in \R$, the functions $f: \R^d \times \R^n \to \R^d$ and $F: \R^d \times \R^n \to \R^n$ are continuously differentiable with locally Lipschitz derivative, and $F(z,0) = 0$ for all $z \in \R^d$. Suppose there exists a function $\bar U: \R^d \times \R^n \to [1,\infty)$ such that:
    \begin{itemize}
        \item $\bar U \in \CC^2$ has compact sublevel sets.
        \item There are constants $K,c > 0$ such that
        \begin{equation}\label{eq:Lyapunov-condition}
            \Ll \bar U \leq K - c\bar U \,,
        \end{equation}
        where $\Ll$ denotes the generator of \eqref{eq:general-sde}.
        \item There is some $C > 0$ such that for all $z \in \R^d$, $x \in \R^n$
        \begin{equation}\label{eq:vanishing-condition}
            |\ip{F(z,x),x}| \leq C|x|^2\Big(2K - \frac{\Ll \bar U}{\bar U} + \frac{\sigma^2|\partial_z \bar U|^2}{\bar U^2}\Big)
        \end{equation}
    \end{itemize}

    Then define the process
    \begin{equation}\label{eq:bigZ-v-process}
        \begin{aligned}
            dZ &= f(Z,0)dt + \sigma dW_t \\
            dv &= [\partial_xF(Z,0)v - \ip{\partial_xF(Z,0)v,v}v]dt \,,
        \end{aligned}
    \end{equation}
    let $P_{\invm}$ denote the set of all invariant measures of $(Z_t, v_t)$ on $\R^d \times S^{n-1}$, and define
    \begin{align*}
        \Lambda^+ &\coloneqq \sup_{\mu \in P_{\invm}} \int \ip{\partial_xF(z,0)v,v}d\mu(z,v) \\
        \Lambda^- &\coloneqq \inf_{\mu \in P_{\invm}} \int \ip{\partial_xF(z,0)v,v}d\mu(z,v) \,.
    \end{align*}

    Denoting $\mcM_0 \coloneqq \{(z,x) \in \R^d \times \R^n \mid x = 0\}$ we conclude:
    \begin{itemize}
        \item If $\Lambda^- > 0$, \eqref{eq:general-sde} has multiple invariant measures, at least one of which gives $\mcM_0$ measure $0$. Furthermore, for all $\delta > 0$ there is some $\eta > 0$ such that for all $(z_0,x_0) \in \inv \coloneqq \mcM_0^c$ we have
        $$\limsup_{t \to \infty} \frac{1}{t}|\{s \leq t \mid |x_s| < \eta \}| < \delta \quad a.s.$$
        \item If $\Lambda^+ < 0$ and there exists $\epsilon > 0$ and a continuously differentiable function $\Upsilon: \R^d \times \R^n \to [0,\infty)$ such that
        \begin{equation}\label{eq:accessibility-condition}
        \partial_z\Upsilon f+\partial_x\Upsilon F
\le -\epsilon\Upsilon \text{ on all of } \R^d \times \R^n 
        \end{equation}
 and if $(z_n,x_n) \in \R^d \times \R^n$ is such that $\Upsilon(z_n,x_n)\to0$ implies  $(z_n,x_n)\to 0$,
        then all invariant measures of \eqref{eq:general-sde} are supported on $\mcM_0$ and
        $$\limsup_{t \to \infty} \frac{1}{t}\ln |x_t| < 0 \quad a.s.$$
    \end{itemize}
\end{thm}

\begin{proof}
    The second bullet point is a consequence of \cite[Theorem 3.5]{extinction}, using \cite[Theorem 3.9, Lemma 9.3]{extinction} to verify \cite[Assumptions 1-5]{extinction} with $V(z,x) = -\ln |x|$, and $\Upsilon$ to verify the accessibility condition. Indeed, \eqref{eq:Lyapunov-condition} and \eqref{eq:vanishing-condition} are exactly the conditions in \cite[Lemma 9.3]{extinction}: $\Gamma V = 0$ since $x$ has no $dW$ term in \eqref{eq:general-sde}, and the explicit form of the generator yields $\Ll V(z,x) =
 f\partial_zV + F \partial_xV + \frac{1}{2} \sigma^2 \Delta_z V =
     -\ip{F(z,x),x}/|x|^2$. For the application of \cite[Theorem 3.9]{extinction}, define $\ncN \coloneqq \R^d \times S^{n-1} \times [0,\infty)$, $\ncN_0 \coloneqq \R^d \times S^{n-1} \times \{0\}$, $\ncN_+ \coloneqq \ncN \setminus \ncN_0 = \R^d \times S^{n-1} \times (0,\infty)$, and $\pi:\ncN \to \R^d \times \R^n$ via $\pi(z,v,r) = (z,rv)$. Then $\pi$ is a quadruple map (\cite[Definition 3.7]{extinction}) from the Markov process $Y_t \coloneqq (z_t,v_t,r_t)$ to $X_t \coloneqq (z_t,x_t)$, where formally $v_t = x_t/|x_t|$ and $r_t = |x_t|$, and rigorously we can define 
    \begin{align*}
        dr &= r\ip{\tilde F(z,v,r),v}dt \\
        dv &= [\tilde F(z,v,r) - \ip{\tilde F(z,v,r),v}v]dt \\
        \tilde F(z,v,r) &\coloneqq \begin{cases}
             r^{-1}F(z,rv) & \text{if } r > 0 \\
    \partial_xF(z,0)v   & \text{if } r = 0 \,.
        \end{cases}
    \end{align*}
    The abuse of notation with \eqref{eq:bigZ-v-process} is intentional because when $r = 0$ (meaning $Y_0 \in \ncN_0$) we have $Y_t = (Z_t,v_t,0)$. Thus, $P_{\invm}$ (as defined above) is exactly $P_{\invm}(\ncN_0)$ (as defined in \cite[Theorem 3.9]{extinction}). Since the continuous extension $H: \ncN \to \R$ of $(\Ll V \circ \pi)(z,v,r) = -\ip{\tilde F(z,v,r),v}$ satisfies $H(z,v,0) = -\ip{\partial_xF(z,0)v,v}$, we conclude that $-\Lambda^+ = \alpha$ (where $\alpha$ is as in \cite[Theorem 3.9]{extinction}). Finally, Gronwall's inequality implies that $\Upsilon(\tilde z_t, \tilde x_t) \to 0$ as $t \to \infty$, where $(\tilde z_t,\tilde x_t)$ solves \eqref{eq:general-sde} with $\sigma = 0$, so the accessibility of $\mcM_0$ (\cite[(3.1)]{extinction}) follows from Stroock-Varadhan support theorem.

    The first bullet point is a consequence of \cite[Theorem 4.4]{persistence} with $V(z,x) = -\ln g(|x|)$, where $g:(0,\infty) \to (0,1]$ is a smooth function with $g(x) = 1$ for $x \geq 1$ and $g(x) = x$ on a neighborhood of $0$. (The discrepancy with $V$ above is only because \cite[Theorem 4.4]{persistence} requires nonnegative $V$. It does not affect any of the conclusions above because $\Ll V$ above and $\Ll V$ here agree exactly for small $|x|$ and in general only differ by a factor of $g'(|x|)|x|/g(|x|)$, which is bounded). Indeed, as above we have that $\Lambda^-$ corresponds exactly to the one in \cite[Definition 4.2]{persistence}. Furthermore, since we already checked that \cite[Assumption 1-3, 4(ii), 5]{extinction} hold above, we have by \cite[Corollary 4.2, Lemma 2.12]{extinction} (to verify the strong law) that \cite[Hypotheses 1-4]{persistence} hold (note that $W'/(1+|H|)$ being proper is implied by $H$ vanishing over $W'$), so \cite[Theorem 4.4]{persistence} indeed applies.
\end{proof}

\bibliography{references}
\bibliographystyle{Martin}

\end{document}